\documentclass{amsart}

\usepackage{todonotes}
\usepackage[mathscr]{eucal}
\usepackage{graphicx}
\usepackage{amscd}
\usepackage{amsmath}
\usepackage{amsthm}
\usepackage{amsxtra}
\usepackage{mathtools}
\usepackage{calc} 
\usepackage{amsfonts}
\usepackage{amssymb}
\usepackage{latexsym}
\usepackage{pdfsync}
\usepackage{mdwlist}

\usepackage{marginnote}
\usepackage[normalem]{ulem} 
\usepackage{cancel} 

\usepackage{tikz-cd}
\usepackage[all]{xy}

\usepackage[colorlinks, bookmarks=true]{hyperref}

\numberwithin{equation}{section}
\newtheorem{theorem}{Theorem}[section]
\theoremstyle{plain}

\newtheorem{corollary}[theorem]{Corollary}
\newtheorem{corollary-definition}[theorem]{Corollary-Definition}

\newtheorem{definition}[theorem]{Definition}
\newcommand{\thistheoremname}{}
\newtheorem*{genericthm*}{\thistheoremname}
\newenvironment{namedthm*}[1]
  {\renewcommand{\thistheoremname}{#1}%
   \begin{genericthm*}}
  {\end{genericthm*}}

\newtheorem{lemma}[theorem]{Lemma}

\newtheorem{proposition}[theorem]{Proposition}

\numberwithin{equation}{section}

\theoremstyle{definition}
\newtheorem{example}[theorem]{Example}
\newtheorem{remark}[theorem]{Remark}

\newcommand{\blank}{\hspace{0.04cm} \rule{2.4mm}{.4pt} \hspace{0.04cm} }
\newcommand{\blankd}{\hspace{0.04cm} \rule{1.5mm}{.4pt} \hspace{0.04cm} }

\DeclareMathOperator{\Ker}{\mathrm{Ker}}
\DeclareMathOperator{\Cok}{\mathrm{Coker}}
\DeclareMathOperator{\coker}{\mathrm{coker}}

\DeclareMathOperator{\Hom}{\mathrm{Hom}}
\DeclareMathOperator{\Ext}{\mathrm{Ext}}
\DeclareMathOperator{\Tor}{\mathrm{Tor}}
\DeclareMathOperator{\Imr}{\mathrm{Im}}

\DeclareMathOperator{\Tr}{\mathrm{Tr}}

\DeclareMathOperator{\ot}{\overset{\rightharpoondown}{\otimes}}

\mathchardef\mhyphen="2D

\newcommand{\Z}{\mathbb{Z}}
\newcommand{\Q}{\mathbb{Q}}

\newcommand{\injt}{\mathfrak{s}}

\newcommand{\ab}{\mathbf{Ab}}

\newcommand{\lra}{\longrightarrow}

\makeatletter
\@namedef{subjclassname@2020}{\textup{2020} Mathematics Subject Classification}
\makeatother

\newcounter{hours}
\newcounter{minutes}

\begin{document}

\title[\today]{Deconstructing Auslander's formulas, I. Fundamental sequences associated with an additive functor.}
\thanks{The author was supported in part by the Shota Rustaveli National Science Foundation of Georgia Grants NFR-18-10849 and FR-24-8249}

\author{Alex Martsinkovsky}
\address{Mathematics Department\\
Northeastern University\\
Boston, MA 02115, USA}
\email{a.martsinkovsky@northeastern.edu}

\date{\today, \setcounter{hours}{\time/60} \setcounter{minutes}{\time-\value{hours}*60} \thehours\,h\ \theminutes\,min}

\dedicatory{In memory of Ragnar-Olaf Buchweitz}

\begin{abstract}
For any additive functor from modules (or, more generally, from an abelian category with enough projectives or injectives), we construct long sequences tying 

together the derived functors, the satellites, and the stabilizations of the functor. For half-exact functors, the obtained sequences are exact. For general functors, nontrivial homology may only appear at the derived functors. Specializing to the familiar Hom and tensor product functors on finitely presented modules, we recover the classical formulas of Auslander. Unlike those formulas, our results hold for arbitrary rings and arbitrary modules, finite or infinite. The same formalism leads to universal coefficient theorems for homology and cohomology without any assumptions. The new results are even more explicit for the cohomology of projective complexes  and the homology of flat complexes. 

\end{abstract}

\maketitle
\tableofcontents

\section{Introduction}

Among the many contributions of Maurice Auslander to mathematics, it is not easy to find anything more ubiquitous than the transpose of a finitely presented module. Given such a module $A$ over a ring $\Lambda$ with finite projective presentation
\[
P_{1} \lra P_{0} \lra A \lra 0,
\]
the transpose $\Tr A$ of $A$ is the $\Lambda^{op}$-module defined by the exact sequence 
\begin{equation}\label{D:TrA}
0 \lra A^{\ast} \lra P_{0}^{\ast} \lra P_{1}^{\ast} \lra \Tr A \lra 0,
\end{equation}
obtained by dualizing the presentation of $A$ into $\Lambda$. This innocuous-looking and elementary construction is fundamental for Auslander-Reiten theory. It also appears in a variety of formulas of a rather general nature. Its primary purpose is to provide a duality on the categories of finitely presented modules modulo projectives  based on the well-known duality for finitely generated projectives. The lack of such duality for infinite projectives makes the transpose for infinite modules useless. This paper introduces a homological formalism that completely bypasses the transpose, produces formulas valid for arbitrary modules over arbitrary rings, and, for finitely presented modules, specializes to Auslander's formulas

In Section~\ref{S:examples}, we look at three motivating examples: the Auslander-Reiten formula and two exact sequences arising from the canonical maps from the tensor product functors to the covariant Hom functors. They all use the transpose and require the fixed argument to be finitely presented. We show that all of those formulas fail when this requirement is dropped. 

The Auslander-Reiten formula is treated in Section~\ref{S:AR-form}. A general result in that direction was already established in~\cite[Remark 9.8]{MR-1}. Here, we briefly review it and state it in a streamlined notation. We also prove a dual result, Proposition~\ref{P:dual}, which doesn't seem to have a classical prototype.

Section~\ref{S:rfs-cov} introduces the notion of the right fundamental sequence associated with an additive covariant functor and establishes the first main result, Theorem~\ref{T:co-fund-right}. It combines the derived functors, the satellites, and the sub-stabilization (formerly known as injective stabilization) of a given functor into a long sequence. For half-exact functors, this sequence is exact. Auslander's result describing the kernel and the cokernel of the canonical natural transformation $A \otimes \blank \lra (A^{\ast}, \blank)$, where $A$ is finitely presented, is recovered as a special case. 

Section~\ref{S:left-her} deals with half-exact finitely presented functors on modules over left-hereditary rings. It should be viewed as a worked example.

In Section~\ref{S:L0}, we establish the existence of left fundamental sequences for covariant functors.  Auslander's result describing the kernel and the cokernel of the canonical natural transformation $A^{\ast} \otimes \blank \lra (A, \blank)$, where $A$ is finitely presented, is recovered as a special case.

Fundamental sequences for contravariant functors and the corresponding theorems are found in Sections~\ref{S:right-fs-contra} and \ref{S:left-fs-contra}.

The first major application of the developed techniques is contained in Section~\ref{S:ucfc}, where we establish a universal coefficient theorem for the cohomology of arbitrary complexes. Classically, this result is known for projective complexes with projective boundaries. Our result holds for arbitrary complexes without any assumptions. Furthermore, this result is sharpened for arbitrary projective complexes. Our approach is based on the fact that the cohomology functor is always finitely presented. 

Section~\ref{S:ucfh} contains another major application: a universal coefficient theorem for the homology of arbitrary complexes. Classically, this result is known for flat complexes with flat boundaries. Our approach allows us to dispose of any assumptions on the complexes. Similar to the cohomological theorems, we have a sharper result for arbitrary flat complexes. 

Unlike their cohomological counterparts, which are based on finitely presented functors, the results for homology cannot use this machinery because tensor products are not always finitely presented. However, the homology functors are always  ``tensor-copresented''. A good way to visualize the analogy between finitely presented and tensor-copresented functors is via the Auslander-Gruson-Jensen (AGJ) transformation, which interchanges the tensor product and the Hom functors. To wit, given an additive functor $F$, the value of the AGJ transform of $F$ on a module~$A$ is defined as the class of natural transformations from $F$ to $A \otimes \blank$. It is not difficult to see that tensor-copresented functors are precisely the AGJ transforms of finitely presented ones. Unfortunately, it is unclear how to use, beyond the heuristics, the AGJ transformation in our context. While the AGJ transformation is exact on the category of finitely presented functors, it fails to be exact on the category of tensor-copresented functors. This is demonstrated in Section~\ref{S:tensor-cop}, where we give a simple counterexample. It is based on another counterexample showing that the tensor product is not an injective object in the category of tensor-copresented functors. Thus, the AGJ transformation is not a duality between finitely presented and tensor-copresented functors. To bypass all these difficulties, we give, in Section~\ref{S:tensor-cop}, a direct construction of the quot-stabilization (formerly known as the projective stabilization) of an arbitrary tensor-copresented functor. This result sets the stage for the universal coefficient theorems proved in Section~\ref{S:ucfh}.
\smallskip

\textbf{Blanket assumptions}. Throughout this paper, $\mathcal{A}$ will denote the category of left (or right, depending on the desired context) modules over an associative ring~$\Lambda$ with identity. More generally, $\mathcal{A}$ may be assumed to be an abelian category with enough projectives or, depending on the context, enough injectives. $\ab$ will denote the category of abelian groups and group homomorphisms. All functors will be assumed to be additive. 

\section{Motivating examples}\label{S:examples}

For motivating examples, we keep the notation from the previous section and revisit several known results. Assume that $\mathcal{A}$ is the category of all (say, left) modules over a ring $\Lambda$ and $A$ is a finitely presented module. In general, 
$\Tr A$ depends on the chosen projective presentation of $A$ and is only well-defined up to projectively stable equivalence in the category of finitely presented modules over $\Lambda^{op}$ or, in short, modulo projectives. If $\Lambda$ is semiperfect, then by choosing a minimal presentation of $A$, one can fix the isomorphism type of $\Tr A$, but even then, the symbol $\Tr$ need not be functorial. This is because an action of $\Tr$ on morphisms would require liftings to the chosen projective presentations, and such liftings are unique only up to homotopy. Consequently, a lifting of a composition need not coincide with the compositions of the liftings.

\begin{example}
 
The well-known Auslander-Reiten formula 
\[
 D_{\mathbf{J}}\Ext^{1}_{\Lambda}(A,B) \simeq \overline{\Hom}_{\Lambda} (B, D_{\mathbf{J}}\Tr A),
 \]  
is a basic tool in Auslander-Reiten theory. Here  $\Lambda$ is an algebra over a commutative ring $\Bbbk$ (which could be $\Z$), $A$ and $B$ are (left) $\Lambda$-modules, $A$ is finitely presented, $\mathbf{J}$ is an injective $\Bbbk$-module, $D_{\mathbf{J}} := \Hom_{\Bbbk}(\blank, \mathbf{J})$, and the transpose is computed using any finite presentation of $A$. The overlined $\Hom$ on the right-hand side stands for $\Hom$ modulo injectives, i.e., the quotient of the usual $\Hom$ by the subgroup of all maps factoring through injectives. It ``neutralizes'' the projective ambiguity in the definition of the transpose, which is turned into an injective ambiguity 
by~$D_{\mathbf{J}}$. If~$A$ is not finitely presented, then  the left-hand side still makes sense, but the right-hand side is not well-defined. This may happen even when $A$ is finitely presented but the transpose is computed via an infinite presentation. In fact, this may happen even for $A \simeq 0$. To see that, assume that $\Lambda$ is not left-coherent. Then, by a theorem of Chase~\cite[Theorem~2.1]{Chase}, 
$\Lambda_{\Lambda}^{N}$ is not flat for some $N$. Now choose the following  projective resolution of the zero module
\[
0 \lra \prescript{}{\Lambda}\Lambda^{(N)} \overset{\mathrm{id}}\lra \prescript{}{\Lambda}\Lambda^{(N)} \lra 0 \lra 0 \lra 0.
\]
Computing the transpose with this resolution, we have $\Tr A \simeq \Lambda_{\Lambda}^{N}$. Now choose $\Bbbk : = \Z$ and $\mathbf{J} := \Q/\Z$. By a theorem of Lambek~\cite[p.~238]{L64}, the character module $D_{\mathbf{J}}\Tr A$ is not injective. Evaluating the right-hand side of the formula on $B := D_{\mathbf{J}}\Tr A$
we have a nonzero group, whereas the left-hand side is zero. Thus the formula fails.
\end{example}

\begin{example}
 Let $A$ be a finitely presented $\Lambda$-module, and $e_{A} : A \lra A^{\ast\ast}$ the canonical evaluation map. Then there is an exact sequence
\begin{equation}\label{Eq:bidual}
\xymatrix
	{
	0 \ar[r]
	& \Ext^{1}(\Tr A, \Lambda) \ar[r]
	& A \ar[r]^{e_{A}}
	& A^{\ast\ast} \ar[r]
	& \Ext^{2}(\Tr A, \Lambda) \ar[r]
	& 0.
	}
\end{equation}
This is the $\Lambda$-component of the exact sequence of functors~\cite[Proposition 2.6 (a)]{AB}
\begin{equation}\label{D:ext}
\xymatrix@R1pt@C15pt
	{
	0 \ar[r]
	& \Ext^{1}(\Tr A, \blank) \ar[r]
	& A \otimes \blank \ar[r]^{\rho_{A}}
	& (A^{\ast}, \blank) \ar[r]
	& \Ext^{2}(\Tr A, \blank) \ar[r]
	& 0.
	}\footnote{The reader may check that  $\rho_{A}$ is the zeroth right-derived transformation for $A \otimes \blank$.}
\end{equation}
When $A$ is not finitely presented, the sequence~\eqref{Eq:bidual} fails to be exact already for vector spaces over a field. Indeed, in that case the end terms vanish, forcing $\rho_{A}$, and hence $e_{A}$, to be an isomorphism, which is impossible since $A^{\ast\ast}$ has a larger dimension than~$A$.
\end{example}

\begin{example}
Let $A$ and $B$ be left $\Lambda$-modules. Next we examine the canonical map 
\[
\varphi : A^{\ast} \otimes B \lra (A, B) : (f, b) \mapsto \big(a \mapsto f(a)b\big).
\]
As $B$ is canonically isomorphic to $(\Lambda, B)$, we interpret $\varphi$ as the map 
\[
(A, \Lambda) \otimes (\Lambda, B) \to (A,B)
\]
given by the ``composition" of the arguments. Therefore the image of 
$\varphi$ consists of all maps $A \to B$ factoring through finitely generated projectives.

Under the assumption that $A$ is finitely presented, we have an exact sequence of functors~\cite[Proposition~7.1]{A66}.
\begin{equation}\label{D:tor}
\xymatrix@R1pt@C15pt
	{
	0 \ar[r]
	& \Tor_{2}(\Tr A, \blank) \ar[r]
	& A^{\ast} \otimes \blank \ar[r]^{\lambda_{A}}
	& (A, \blank) \ar[r]
	& \Tor_{1}(\Tr A, \blank) \ar[r]
	& 0.
	}\footnote{The reader may check that $\lambda_{A}$ is the zeroth left-derived transformation for $(A, \blank)$.}
\end{equation}

However, when $A$ is not finitely generated, this sequence fails to be exact already for vector spaces over a field. In that case, the end terms vanish, and the exactness of the sequence would force $\lambda_{A}$ to be an isomorphism, which is impossible since, evaluating it on $B : = A$, we would have the identity map on $A$ factor through a finite-dimensional vector space, a contradiction. 
\end{example}

\section{The right fundamental sequence of a covariant functor}\label{S:left-fs-co}\label{S:rfs-cov}

The main tools that will give us the desired general formulas are the injective (respectively, projective) stabilization of an additive functor, which, roughly speaking, ``measure the deviation'' of the functor from its zeroth derived functor. These concepts are already present in~\cite{A66} and~\cite{AB}, and were substantially expounded in~\cite{MR-1, MR-2, MR-3}. One may think that the adjectives ``injective'' and ``projective'' refer to the choice of resolutions. It is formally true for covariant functors but  false for contravariant ones, which leads to considerable confusion in practical situations. Thus, for psychological reasons, we have changed the terminology and instead speak of ``sub-stabilizations'' and, respectively, ``quot-stabilizations''.

\subsection{The right fundamental sequence}

Suppose $\mathcal{A}$ has enough injectives. Then, given an additive functor 
${F} : \mathcal{A} \lra \ab$, the right-derived functors $R^{i}F : \mathcal{A} \lra \ab$ can be defined for all $i \in \Z$. To find $R^{i}F(B)$ for $B \in \mathcal{A}$, pick an injective resolution of $B$, apply $F$ to it, and define $R^{i}F(B)$ as the $-i$th homology of the resulting complex. The universal property of kernels yields a map $F(B) \lra R^{0}F(B)$ functorial in $B$. The kernel $\overline{F}$ of the corresponding  natural transformation 
$\rho_{F} : F \lra R^{0}F$ is called the \texttt{sub-stabilization} of $F$. Thus, we have a defining exact sequence
\begin{equation}\label{E:first-def-seq}
0 \lra \overline{F} \lra F \overset{\rho_{F}}\lra R^{0}F.
\end{equation}

As $\rho_{F}$ is clearly an isomorphism on injectives, $\overline{F}$ vanishes on injectives.\footnote{In fact, $\overline{F}$ is the largest subfunctor of $F$ with this property. Since $\mathcal{A}$ is assumed to have enough injectives, $\overline{F}$ can also be characterized as the largest effaceable subfunctor of $F$.} It now follows from Schanuel's lemma, that $\overline{F} \circ \Sigma$ is a well-defined functor, where~$\Sigma$ denotes the operation of taking the cokernel of the inclusion of a module in an injective container. For brevity, we shall refer to that cokernel as a cosyzygy module. 

Equivalently, given an object $B$, $\overline{F}(B)$ can be defined and computed by embedding $B$ in an injective container $0 \lra B \overset{\iota} \lra I$, applying $F$, and taking the kernel. This yields a system of defining exact sequences
\begin{equation}\label{E:i-stab-via-container}
 0 \lra \overline{F}(B) \overset{k}\lra {F}(B)  \overset{F(\iota)}\lra F(I),
\end{equation}
where, to avoid notational clutter, we have stripped $k$ of any mention of $F$ and $B$.

Next, we want to construct a natural transformation 
$\beta_{F} : R^{0}F \lra \overline{F} \circ \Sigma$. To this end, it is convenient to use the $\Ker$-$\Cok$ exact sequence (see C.~H.~Dowker~\cite[Theorem 1]{D} and J.~B.~Leicht~\cite[Satz 9]{Leicht64}), defined for any pair of composable morphisms.
To avoid confusing it with the $\Ker$-$\Cok$ sequence of the snake lemma, we shall use the term ``circular sequence''. 

To define $\beta_{F}(B)$, start with an injective resolution 
\[
0 \lra B \overset{i} \lra I^{0} \overset{d_{0}}\lra I^{1} \overset{d_{1}}\lra I^{2} \lra \ldots \quad
\] 
Let $e : I^{0} \lra \Sigma B$ be the cokernel of $i$ and 
\[
\xymatrix
	{
	I^{0}  \ar@{->>}[rd]_{e} \ar[rr]^{d_{0}}
	&
	& I^{1} 
\\
	& \Sigma B \ar@{>->}[ru]_{m}
	&
	}
\]
the corresponding epi-mono factorization of $d_{0}$. Recall that the first right satellite $S^{1}F$ is defined as the functor whose value at $B$ is $\Cok F(e)$.
It is easy to see that it is well-defined and additive.

Applying $F$ and passing to the associated circular diagram
\begin{equation}\label{D:circ-diag}
\begin{tikzcd}
 0 \ar[d]
 & \overline{F}(B) \ar[d, tail, "k"]
 &
 &
 & 0
\\
 K \ar[d, rightarrowtail, "a"'] \ar[dr, tail, "j"'] 
 & F(B) \ar[d, "F(i)"] \ar[l, "z"', dashrightarrow]
 &
 &
 & L \ar[u]
\\
 R^{0}F(B) \ar[r, tail] \ar[ddr, "\beta_{F}(B)"']
 & F(I^{0}) \ar[rd, "F(e)"'] \ar[rr, "F(d_{0})"]
 &
 & F(I^{1}) \ar[r, two heads] \ar[ru, two heads]
 & B_{2} \ar[u, "l"']
\\
 &
 & F(\Sigma B) \ar[ru, "F(m)"'] \ar[rd, two heads]
 \\
 &\overline{F}(\Sigma B) \ar[ru, tail] \ar[rr, "g"']
 &
 & S^{1}F(B) \ar[ruu, "h"']
\end{tikzcd}
\end{equation}
where $K := \Ker F(e)$, $B_{2} := \Cok F(d_{0})$, $L := \Cok F(m)$, and $S^{1}F$ denotes the first right satellite of $F$, we have a six-term exact sequence 
\begin{equation}\label{D:6-term-es}
 0 \lra K \lra R^{0}F(B) \lra \overline{F}(\Sigma B) \lra S^{1}F(B) \lra B_{2} \lra L \lra 0
\end{equation} 

The universal property of kernels yields the dashed arrow
$z : F(B) \to K$, lifting $F(i)$ against $j$. Plainly, $az : F(B) \lra R^{0}F(B)$ is the canonical map afforded by the universal property of kernels, and, including it in the sequence, we have a seven-term sequence
 \begin{equation}\label{D:7-term-cx}
\begin{tikzcd}[cramped, sep=small]
 0 \ar[r]
 & \overline{F}(B) \ar[r, "k"] 
 & F(B) \ar[r, "az"]  
 & R^{0}F(B) \ar[r] 
 & \overline{F}(\Sigma B) \ar[r]
 & S^{1}F(B) \ar[r]
 & B_{2} \ar[r]
 & L \ar[r] 
 & 0.
 \end{tikzcd}
\end{equation}
Clearly, the sequence is functorial in $B$ at the first five terms. In particular, we have the desired natural transformation $\beta_{F} : R^{0}F \lra \overline{F} \circ \Sigma$. It is also clear that the isomorphism types of the first five terms are determined uniquely, but this cannot be said about $B_{2}$ and $L$, which depend on the choice of the resolution. Our next goal is to find a functorial replacement for them. This will be done by ``extending'' the current sequence.

Applying $F$ to the epi-mono factorization $d_{1} = m_{1}e_{1}$ and combining the result with a slightly redrawn diagram~\eqref{D:circ-diag}, we have a commutative diagram of solid arrows:

\begin{equation}\label{D:circ-diag-extended}
\begin{tikzcd}
 &
 & R^{0}F(B) \ar[d, rightarrowtail] \ar[ldd, dashrightarrow] \ar[dll, "\beta_{F}"']
 & F(B) \ar[l] \ar[ld]
 & 
\\
 \overline{F}(\Sigma B) \ar[dd] \ar[dr, tail, "\iota"']
 &
 & F(I^{0}) \ar[dl, "F(e)" description, pos=0.4] \ar[dd, "F(d_{0})"', pos=0.37] \ar[r]
 & Z_{1} \ar[r, two heads] \ar[ddl, tail, "\kappa_{1}"', pos=0.6] \ar[ddddlll, blue, dashrightarrow]
 & R^{1}F(B) \ar[ddddllll, blue, dashrightarrow]
\\
 & F(\Sigma B) \ar[dl, two heads, "\pi"'] \ar[dr, "F(m)"', pos=0.4] \ar[rru, magenta, dashrightarrow]
 &
 &
 &
\\
 S^{1}F(B) \ar[rrrruu, magenta, dashrightarrow, "\eta", pos=0.4]
 &
 & F(I^{1})  \ar[ldd, "F(e_{1})"']  \ar[dddd, "F(d_{1})"', pos=0.35]  \ar[r]
 & Z_{2} \ar[r] \ar[ldddd, tail, "\kappa_{2}"', pos=0.6]
 & R^{2}F(B)
\\
\\
 \overline{F}(\Sigma^{2} B) \ar[r, tail] \ar[dd]
 &  F(\Sigma^{2} B) \ar[ddr, "F(m_{1})"', pos=0.8] \ar[ddl, two heads] \ar[rruu, magenta, dashrightarrow]
\\
\\
  S^{1}F(\Sigma B) \ar[rrrruuuu, magenta, dashrightarrow]
  &
  & F(I^{2})
\end{tikzcd}
\end{equation}
where $\kappa_{i} : Z_{i} \lra F(I^{i})$ is the kernel of $F(d_{i})$, $i=1,2$, and  
$Z_{i} \lra R^{i}F(B)$ is the cokernel of $F(I^{i-1}) \lra Z_{i}$, $i=1,2$. 

Now, we define the dashed arrows. The diagonal 
$R^{0}F(B) \lra {F}(\Sigma B)$ just expresses the commutativity of the ambient square. Since $e_{1}m = 0$, the universal property of kernels implies that $F(\Sigma B) \lra F(I^{1})$ factors through $\kappa_{1}$, which yields the dashed arrow $F(\Sigma B) \lra Z_{1}$. Precomposing with the monomorphism $\kappa_{1}$ shows that the triangle $F(I^{0}), F(\Sigma B), Z_{1}$ commutes, which allows to define the dashed arrow $\eta : S^{1}F(B) \lra R^{1}F(B)$ as the induced map on the relevant cokernels.

To define the next two arrows, notice that $F(e_{1})\kappa_{1}$ factors through  the kernel of $F(m_{1})$, which yields the arrow $Z_{1} \lra \overline{F}(\Sigma^{2} B)$. It is clear that the composition 
\[
F(I^{0}) \lra Z_{1} \lra \overline{F}(\Sigma^{2} B) \lra F(\Sigma^{2} B)
\] 
is zero, and, since the last map is monic, the middle map extends to 
$R^{1}F(B)$, which gives the next arrow. The remaining two arrows are defined similarly to the first two.

Splicing the first five terms of~\eqref{D:7-term-cx} with the matching part of the peripheral sequence in~\eqref{D:circ-diag-extended}, and continuing further down the diagram in the manner just described, we have a sequence

\begin{equation}\label{D:rfs-co}
\begin{tikzcd}
 0 \ar[r] 
   & \overline{F}(B) \ar[r] 
      & F(B) \ar[r] 
                  \ar[d, phantom, ""{coordinate, name=Z}]
         & R^{0}F(B) \ar[dll, rounded corners,
                                   to path={ -- ([xshift=2ex]\tikztostart.east)
						|- (Z) [near end]\tikztonodes
						-| ([xshift=-2ex]\tikztotarget.west) 
						-- (\tikztotarget)}] 
\\
   & \overline{F}(\Sigma B) \ar[r] 
      & S^{1}F(B) \ar[r]
      			  \ar[d, phantom, ""{coordinate, name=Y}]
         & R^{1}F(B) \ar[dll, rounded corners,
                                   to path={ -- ([xshift=2ex]\tikztostart.east)
						|- (Y) [near end]\tikztonodes
						-| ([xshift=-2ex]\tikztotarget.west) 
						-- (\tikztotarget)}] 
\\
   & \overline{F}(\Sigma^{2} B) \ar[r] 
      & S^{1}F(\Sigma B) \ar[r]
      			  \ar[d, phantom, ""{coordinate, name=X}]
         & R^{1}F(\Sigma B) \ar[dll, rounded corners,
                                   to path={ -- ([xshift=2ex]\tikztostart.east)
						|- (X) [near end]\tikztonodes
						-| ([xshift=-2ex]\tikztotarget.west) 
						-- (\tikztotarget)}]
\\
   & \overline{F}(\Sigma^{3} B) \ar[r] 
      & S^{1}F(\Sigma^{2} B) \ar[r]
         & \ldots 
\end{tikzcd}
\end{equation}
which is functorial in $B$. Utilizing dimension shift, we have a sequence of functors
\begin{equation}\label{D:rfs-co-fun}
\begin{tikzcd}
 0 \ar[r] 
   & \overline{F} \ar[r] 
      & F \ar[r, "\rho_{F}"] 
                  \ar[d, phantom, ""{coordinate, name=Z}]
         & R^{0}F \ar[dll, "\beta_{F}", rounded corners,
                                   to path={ -- ([xshift=2ex]\tikztostart.east)
						|- (Z) [pos=0.3]\tikztonodes
						-| ([xshift=-2ex]\tikztotarget.west) 
						-- (\tikztotarget)}] 
\\
   & \overline{F} \circ \Sigma \ar[r] 
      & S^{1}F \ar[r]
      			  \ar[d, phantom, ""{coordinate, name=Y}]
         & R^{1}F \ar[dll, rounded corners,
                                   to path={ -- ([xshift=2ex]\tikztostart.east)
						|- (Y) [near end]\tikztonodes
						-| ([xshift=-2ex]\tikztotarget.west) 
						-- (\tikztotarget)}] 
\\
   & \overline{F} \circ \Sigma^{2} \ar[r] 
      & S^{2}F  \ar[r]
      			  \ar[d, phantom, ""{coordinate, name=X}]
         & R^{2}F \ar[dll, rounded corners,
                                   to path={ -- ([xshift=2ex]\tikztostart.east)
						|- (X) [near end]\tikztonodes
						-| ([xshift=-2ex]\tikztotarget.west) 
						-- (\tikztotarget)}]
\\
   & \overline{F} \circ \Sigma^{3} \ar[r] 
      & S^{3}F \ar[r]
         & \ldots
\end{tikzcd}
\end{equation}


\begin{definition}
 We shall call~\eqref{D:rfs-co-fun} the right fundamental sequence of $F$.
\end{definition}

Our next goal is to understand the exactness properties of the fundamental sequence. The circular sequence and the following obvious lemma will be  useful for that purpose.

\begin{namedthm*}{The Parallelogram Lemma}
Suppose, in the commutative diagram 
\[
\begin{tikzcd}
 A \ar[r] \ar[rd]
 & B \ar[d, "f"'] \ar[rd]
 \\
 & C \ar[r]
 & D
\end{tikzcd}
\]
the two peripheral compositions are zero.

\begin{enumerate}
 \item [a)] If $f$ is epic and the upper peripheral sequence $A \to B \to D$ is exact at $B$, then the lower peripheral sequence $A \to C \to D$ is exact at $C$.
 \item [b)] If $f$ is monic and the lower peripheral sequence $A \to C \to D$ is exact at $C$, then the upper peripheral sequence $A \to B \to D$ is exact at $B$.
\end{enumerate}
\end{namedthm*}

\begin{proof}
 Obvious.
\end{proof}

\begin{theorem}\hfill\label{T:co-fund-right}
\begin{enumerate}
 \item The right fundamental sequence~\eqref{D:rfs-co-fun} is a complex of functors, whose  homology is zero except, possibly, at the terms $R^{i}F$.
  
 \item If $F$ is half-exact, then the sequence is exact. 
 
 \item If $F$ is epi-preserving, then $\beta_{F}$ is epic.
 
 \item If $F$ is mono-preserving, then $\overline{F} = 0$ and $\rho_{F}$ is monic.
 
 \item $\rho_{F}$ is an isomorphism if and only if $F$ is left-exact.

\end{enumerate}
\end{theorem}

\begin{proof}
(1) The claim that~~\eqref{D:rfs-co-fun} is a complex is easily checked by a diagram chase of arrows (with or without elements) in~\eqref{D:circ-diag-extended} and is left to the reader as an exercise. 

For the exactness claim, we have to show the exactness at the first two terms of each row. For the first row, this is true by the definition of the injective stabilization. 

For the second row, since $\kappa_{1}$ is monic, it is easy to check that the fragment of~\eqref{D:rfs-co} connecting $R^{0}F$ and $R^{1}F$ is part of the circular sequence of the triangle $F(I^{0}), F(\Sigma B), Z_{1}$ in the diagram~\eqref{D:circ-diag-extended}, and we are done by the exactness of the circular sequence. 

For the third row, since $\kappa_{2}$ is monic, we have the fragment 
\[
Z_{1} \lra \overline{F}(\Sigma^{2} B) \lra S^{1}F(\Sigma B) \lra R^{2}F(B)
\] 
of the circular sequence of the triangle $F(I^{1}), F(\Sigma^{2} B), Z^{2}$. Since 
$Z_{1} \lra R^{1}F(B)$ is epic, these two modules have the same image in 
$\overline{F}(\Sigma^{2} B)$, and this yields the desired result for the third row. The same argument applies to the remaining rows.

(2) The top part of~\eqref{D:circ-diag-extended} yields the commutative diagram
\[
\begin{tikzcd}
 F(B) \ar[r] \ar[rd]
 & R^{0}F(B) \ar[d, tail] \ar[rd, "\iota \beta_{F}"] \ar[r, "\beta_{F}"]
 & \overline{F}(\Sigma B) \ar[d, "\iota", tail]
 \\
 & F(I^{0}) \ar[r]
 & F(\Sigma B)
\end{tikzcd}
\]
If $F$ is half-exact, the lower peripheral sequence is exact at $F(I^{0})$. By the parallelogram lemma, b), the top peripheral sequence is exact at $R^{0}F(B)$. Since $\iota$ is monic, the top row is also exact at that spot, which is the desired claim for $R^{0}F$.

To prove the exactness at $R^{1}F$, we examine the commutative diagram
\[
\begin{tikzcd}
 F(\Sigma B) \ar[r] \ar[rd, "\eta \pi"] \ar[d, two heads, "\pi"']
 & Z_{1} \ar[d, two heads] \ar[rd]
 \\
 S^{1}F(B) \ar[r, "\eta"']
 & R^{1}F(B) \ar[r]
 & \overline{F}(\Sigma^{2} B)
\end{tikzcd}
\]
The circular sequence for the composable pair $F(e_{1}), F(m_{1})$ produces 
a sequence 
\[
\Ker F(e_{1}) \lra Z_{1} \lra \overline{F}(\Sigma^{2} B),
\]
which is exact at $Z_{1}$. If $F$ is half-exact, then $\Ker F(e_{1}) = \Imr F(m)$, which shows that the top peripheral sequence in the parallelogram above becomes exact at $Z_{1}$. By the parallelogram lemma, a), the lower peripheral sequence is exact at $R^{1}F(B)$. Since $\pi$  is epic, the bottom row is exact at the same spot, which is the desired claim for $R^{1}F$. The same argument applies to the remaining $R^{i}F$.

(3) If $F$ is epi-preserving, then $S^{1}F = 0$, and we are done by the exactness of the circular sequence of the diagram~\eqref{D:circ-diag} at $\overline{F}(\Sigma B)$.

(4) If $F$ is mono-preserving, then $\overline{F} = 0$, and we are done by (1).

(5) This is well-known.
\end{proof}

\begin{corollary}\label{C:connect-iso}
If $F$ is a right-exact functor, then the canonical natural transformation $R^{i}F \lra \overline{F} \circ \Sigma^{i+1}$ is an isomorphism for each $i \geq 1$. 
\end{corollary}

\begin{proof}
 Since $F$ is right-exact, $S^{i}F = 0$ for each $i \geq 1$. The result now follows from the exactness of the fundamental sequence.
\end{proof}

\begin{remark}
 Suppose $F$ is left-exact but not exact. Then the fundamental sequence is exact by~(2) and $\rho_{F}$ is an isomorphism by~(5). It follows that the converse of (3) is not true in general.
\end{remark}

\begin{example}\label{EG:tensor-product}
Continue to assume that $\mathscr{A}$ is the category of left $\Lambda$-modules over a ring~$\Lambda$. Let $A$ be a right $\Lambda$-module and $F := A \otimes \blank$. The sub-stabilization of $A \otimes \blank$ is denoted by $A \ot \blank$ (see~\cite{MR-3, MR-2, MR-1}, where it was called the injective stabilization of the tensor product, for an extensive treatment of this concept). The initial segment of the right fundamental sequence~\eqref{D:rfs-co-fun} becomes
 \begin{equation}\label{Eq: 4-term for tensor product}
{0 \lra A \ot \blank \lra A \otimes \blank \lra R^{0}(A \otimes \blank)  \lra A \ot (\Sigma \blank) \lra 0.}
\end{equation}
By the preceding theorem, it is exact. Evaluating it on  $\prescript{}{\Lambda}\Lambda$, we have another exact sequence (in the category of right $\Lambda$-modules)
 \[
 0 \lra \injt(A) \lra A  \lra R^{0}(A \otimes \blank)(\Lambda) \lra A \ot \Sigma \Lambda \lra 0,
 \] 
 where $\injt$ denotes the torsion radical introduced in~\cite{MR-2}.
\end{example}
\begin{corollary}
 Suppose $\Lambda$ is left hereditary. Then $R^{0}(A \otimes \blank) \simeq I(A\otimes \blank)$, the quotient functor of $A \otimes \blank$ by the largest subfunctor vanishing on injectives, and the sequence~\eqref{Eq: 4-term for tensor product} becomes
 \[
 0 \lra A \ot \blank \lra A \otimes \blank \lra I(A \otimes \blank) \lra  0.
 \]
 In particular, $I(A \otimes \blank)$ is an exact functor.\footnote{The functor $I(A \otimes \blank)$ is not as mysterious as it may seem. Without any assumption 
 on~$\Lambda$ or on $A$, its value on a left module $B$ can be computed by applying $A\otimes \blank$ to a cosyzygy sequence 
 $0 \to B \to I \to \Sigma B \to 0$ and taking the kernel of the induced map $A \otimes I \to A \otimes \Sigma B$.}
\end{corollary}
\begin{proof}
  Since the cosyzygy module of any module is injective, $A \ot (\Sigma_{\blank}) = 0$. Since the fundamental sequence is exact, we then have $R^{0}(A \otimes \blank) \simeq I(A \otimes \blank)$. To show that  $I(A \otimes \blank)$ is exact, we first notice that it preserves epimorphisms because it is a quotient of the right-exact functor $A \otimes \blank$. Being a zeroth right-derived functor, it is also left-exact, and therefore exact.
\end{proof}

\subsection{The case of a finitely presented functor}\label{SS:fpf}
Next, we want to impose conditions on the functor $F$ in order to obtain more precise information on the right fundamental sequence. We start with the assumption that $F$ is finitely presented, i.e., there is a module homomorphism   $f: A \to B$ and an exact sequence
\[
(B, \blank) \overset{(f, \blankd)} \lra (A, \blank) \lra F \lra 0.
\]
Notice that, by Yoneda's lemma, any morphism between the first two terms is induced by some $f : A \lra B$. The following construct (without a name) was introduced in~\cite{A66}.
\begin{definition}
 $w(F) := \Ker f$ is called the \texttt{defect} of $F$.
\end{definition}
The defining exact sequence
\[
0 \lra w(F) \lra A \overset{f}\lra B,
\]
gives rise to a commutative diagram 
\[
\xymatrix
	{
	(B, \blank) \ar[r]^{(f, \blank)}
	& (A, \blank) \ar[r] \ar[rd]
	& F \ar[r] \ar@{-->}[d]^{\rho_{F}}
	& 0
\\
	&
	& (w(F), \blank),
	}
\] 
where the dashed arrow is defined by the universal property of cokernels. By the uniqueness of cokernels, $\rho_{F}$ is an isomorphism on any injective module. As $(w(F), \blank)$ is left-exact, the universal property of the zeroth right-derived functors shows that~$\rho_{F}$ is the zeroth right-derived transformation of $F$ and 
\[
R^{0}F \simeq (w(F), \blank).
\]
The last isomorphism and the Yoneda lemma show that the isomorphism class of $w(F)$ is independent of the chosen presentation of $F$. As another consequence, the right fundamental sequence~\eqref{D:7-term-cx} rewrites as
\begin{equation}\label{D:7-term-cx-fp}
\begin{tikzcd}
 0 \ar[r]
 & \overline{F} \ar[r, "k"] 
 & F \ar[r, "\rho_{F}"]  
 & (w(F), \blank) \ar[r, "\beta_{F}"] 
 & \overline{F} \circ \Sigma  \ar[r]
 & S^{1}F \ar[r]
 & \ldots
 \end{tikzcd}
\end{equation}

Since the $R^{i}F$ are defined in terms of injectives and 
$\rho_{F} : F \to R^{0}F$ is an isomorphism on injectives, $R^{i} \rho_{F}$ is an isomorphism for all $i$. Thus we have

\begin{proposition}
 If $F$ is a finitely presented covariant functor, then 
 \[
 R^{i} \rho_{F} : R^{i}F \lra \Ext^{i}(w(F), \blank)
 \]
is an isomorphism for all $i$. \qed
\end{proposition}

\begin{example}
 We again set $F := A \otimes \blank$ and assume that $A$ is \texttt{finitely presented}, i.e., there is an exact sequence 
 \[
 P_{1} \overset{d}\lra P_{0} \overset{p}\lra A \lra 0,
 \]
with $P_{0}$ and $P_{1}$ finitely generated projectives. This sequence gives rise to an exact sequence of functors
 \[
 P_{1} \otimes \blank \overset{d \otimes \blankd}\lra P_{0} \otimes \blank \lra A\otimes \blank \lra 0,
 \]
which is isomorphic to the sequence
\[
 (P_{1}^{\ast}, \blank) \overset{(d^{\ast}, \blankd)}\lra (P_{0}^{\ast}, \blank) \lra A \otimes \blank \lra 0.
 \]
Thus $A \otimes \blank$ is finitely presented.
To compute its defect, we use the defining exact sequence for the transpose of $A$
\begin{equation}\label{D:Tr}
0 \lra A^{\ast} \overset{p^{\ast}}\lra P_{0}^{\ast} \overset{d^{\ast}}\lra P_{1}^{\ast} \lra \Tr A \lra 0,
\end{equation}
which shows that $w(A \otimes \blank) \simeq A^{\ast}$. As a result, the exact sequence~\eqref{Eq: 4-term for tensor product} becomes
\begin{equation}\label{Eq:4-term-tensor-half-explicit}
\xymatrix
	{
	0 \ar[r]
	& A \ot \blank \ar[r]
	& A \otimes \blank \ar[r]^{\rho_{A \otimes \blankd}}
	& (A^{\ast}, \blank) \ar[r]
	& A \ot \Sigma \blank \ar[r]
	& 0.
	}
\end{equation}
One can also check that $\rho_{A \otimes \blank}$ evaluates on a module $B$ according to the rule
\[
a \otimes b \mapsto l : f \mapsto f(a)b.
\]
To further explicate the structure of \eqref{Eq:4-term-tensor-half-explicit}, we compute $A \ot \blank$. The foregoing calculations produce a commutative diagram with an exact row
\[
\xymatrix
	{
	0 \ar[r]
	& (\Tr A, \blank) \ar[r]
	& (P_{1}^{\ast}, \blank) \ar[r]^{(d^{\ast}, \blankd)}
	& (P_{0}^{\ast}, \blank) \ar[r]^{\pi} \ar[rd]_{(p^{\ast}, \blankd)}
	& A \otimes \blank \ar[r] \ar[d]^{\rho_{A \otimes \blankd}}
	& 0
\\
	&
	&
	&
	& (A^{\ast}, \blank)
	}
\]
The circular sequence for the composable morphisms $\pi$ and $\rho_{A \otimes \blankd}$ yields a short exact sequence 
\[
0 \lra \Ker \pi \lra \Ker (p^{\ast}, \blank) \lra 
\Ker \rho_{A \otimes \blankd} \lra 0
\]
Notice that $\Ker \pi = \Imr (d^{\ast}, \blank)$ and $\Ker \rho_{A \otimes \blankd} = A \ot \blank$. To identify the middle term, choose an epimorphism $q : Q \twoheadrightarrow A^{\ast}$ with $Q$ projective. This gives the first three terms of a projective resolution 
\begin{equation}
\ldots \lra Q \overset{p^{\ast}q}\lra P_{0}^{\ast} \overset{d^{\ast}}\lra P_{1}^{\ast} \lra \Tr A \lra 0
\end{equation}
of  $\Tr A$, which is enough for computing $\Ext^{1}(\Tr A, \blank)$. Since $(q, \blank)$ is monic, $\Ker(p^{\ast}q, \blank) = \Ker(p^{\ast}, \blank)$, and we recover the well-known isomorphism.
\[
\Ext^{1}(\Tr A, \blank) \simeq \Ker (p^{\ast}, \blank)/\Imr (d^{\ast}, \blankd) \simeq A \ot \blank.
\]
Finally, using dimension shift on the covariant argument of $\Ext$, we rewrite~\eqref{Eq:4-term-tensor-half-explicit} in a more explicit form
\begin{equation}\label{Eq:4-term-tensor-explicit}
0 \lra \Ext^{1}(\Tr A, \blank) \lra A \otimes \blank 
\overset{\rho_{A \otimes \blankd}}\lra (A^{\ast}, \blank) \lra
\Ext^{2}(\Tr A, \blank) \lra 0,
\end{equation}
thus recovering Auslander's exact sequence~\eqref{D:ext}.
\end{example}

\section{Example: half-exact finitely presented functors on modules over left-hereditary rings}\label{S:left-her}

In this section, we assume  that $\Lambda$ if left-hereditary, and, 
for illustrative purposes, give a streamlined exposition of~\cite[Corollary~4.14]{A66}. 

Since a cosyzygy module of any module is now injective, and, since the sub-stabilization of a covariant functor vanishes on injectives, the right fundamental sequence~\eqref{D:7-term-cx} rewrites as
\begin{equation}
\xymatrix
	{
	0 \ar[r]
	& \overline{F} \ar[r]
	& F \ar[r]^{\rho_{F}}
	& R^{0}F \ar[r]
	& 0.
	}
\end{equation}
Assume, furthermore, that $F$ is half-exact. In that case, Theorem~\ref{T:co-fund-right}, (2) shows that this sequence is exact. Finally, assume that $F$ is finitely presented. Then $R^{0}F \simeq (w(F), \blank)$ and we have a short exact sequence 
\begin{equation}\label{D:ses}
\xymatrix
	{
	0 \ar[r]
	& \overline{F} \ar[r]
	& F \ar[r]^>>>>{\rho_{F}}
	& (w(F), \blank) \ar[r]
	& 0.
	}
\end{equation}
We remark that $\overline{F}$, being the kernel of a natural transformation between finitely presented functors, is itself finitely presented.\footnote{For a detailed proof of this fact, see~\cite[Proposition~3.1, 2)]{A82}.} Thus, we have a short exact sequence in the abelian category of finitely presented functors. Since the last term of the sequence is projective (by Yoneda's lemma), it splits off and we have a direct sum decomposition $F \simeq \overline{F} \coprod (w(F), \blank)$. Notice that $F$ determines both summands uniquely up to isomorphism. 

We also recall that the sub-stabilization of a half-exact functor is itself half-exact.\footnote{For a quick direct proof of this fact, see~\cite[Lemma 7.1]{MR-1}. For a more conceptual proof, see~\cite[Remark~(1.9)]{AB}. } It is not difficult to see that a finitely presented functor vanishes on injectives if and only if its defect is zero. Thus the defect of $\overline{F}$ is zero. Another result of Auslander~\cite[Proposition~4.8]{A66} then shows that $\overline{F}$ is isomorphic to $\Ext^{1}(D,\blank)$ for a suitable module~$D$. As a consequence, any half-exact finitely presented functor  $F$ on the category of left $\Lambda$-modules is of the form 
$\Ext^{1}(D, \blank) \coprod (E, \blank)$, where $E \simeq w(F)$. If a presentation of $F$ is known, say  
\[
(B, \blank) \overset{(f, \blank)} \lra (A, \blank) \lra F \lra 0,
\]
then we can determine a projective resolution of $\overline{F} \simeq \Ext^{1}(D, \blank)$, without knowing the module $D$. To this end, compute the epi-mono factorization of $f$ together with its kernel and cokernel: 
\begin{equation}\label{D:defect}
\begin{gathered}
\xymatrix
	{
	0 \ar[r]
	& w(F) \ar[r]^{l}
	& A \ar[rr]^{f} \ar@{->>}[rd]^{p}
	&
	& B \ar[r]
	& C \ar[r]
	& 0
\\
	& 
	&
	& \Imr f \ar@{>->}[ru]^{i}
	}
\end{gathered}
\end{equation}
Then 
\begin{equation}\label{D:Fsub}
\xymatrix
	{
	0 \ar[r]
	& (C,-) \ar[r]
	& (B,\blank) \ar[r]^{(i, \blank)} 
	& (\Imr f, \blank) \ar[r] 
	& \overline{F} \ar[r]
	& 0,
	}
\end{equation}
is a projective resolution of $\overline{F}$ (see, for example,~\cite[Proposition~6.1, (3)]{MR-1}).\footnote{This result holds for any finitely presented functor over any ring.}

To end this section, we  address the question of the uniqueness of the  decomposition of $F$. The uniqueness does hold and is a consequence of the following general result~\cite[Proposition~4.19]{MR-1}. Before stating it, recall that a functor is said to be injectively stable if it vanishes on injectives. We will be looking at the (large) category of additive functors on an abelian category with values in abelian groups. Assume also that the category in question has enough injectives. We then have

\begin{proposition}
 Let $\mathscr{T}$ be the class of injectively stable covariant functors and 
 $\mathscr{F}$ the class of mono-preserving covariant functors. Then 
 $(\mathscr{T}, \mathscr{F})$ is a hereditary torsion theory in the (large) category of additive covariant functors. \qed
\end{proposition}
Thus both summands in the decomposition of $F$ are uniquely determined by $F$ in the following sense: if $F$ is a direct sum of an injectively stable functor and a representable functor, then the injectively stable functor is isomorphic to 
$\overline{F}$ and the representable functor is isomorphic to $(w(F), \blank)$.

It remains to determine how unique are the modules $D$ and $E$ in the decomposition $F \simeq \Ext^{1}(D, \blank) \coprod (E, \blank)$. As we just saw, 
$(E, \blank) \simeq \big(w(F), \blank\big)$, and therefore, $E \simeq w(F)$.  On the other hand, the Hilton-Rees theorem shows that $D$ is unique up to projectively stable equivalence.

\begin{remark}
 Returning to an arbitrary (i.e., not necessarily left-hereditary) ring, we  have, by~\cite[Proposition~4.3]{A66}, that defect-zero half-exact finitely presented functors  on the module category are precisely the injective objects in the category of defect-zero finitely presented functors. We also have [idem.] that such functors are precisely  direct summands of the functors $\Ext^{1}(D, \blank)$. It is interesting to revisit this result from a constructive point of view. Over a left hereditary ring, if we know a short exact sequence $0 \to A \to B \to C \to 0$  which gives rise to a defining exact sequence for $F$, 
 \[
 0 \lra (C, \blank) \lra (B, \blank) \lra (A, \blank) \lra F \lra 0,
 \]
then~\cite[Lemma 4.1 and Proposition 4.3]{A66} show that we can set $D := C$. In general, we may ask if there is a ``coordinate-free'' description of $D$. For example, the foregoing argument and the universal coefficient formula for cohomology (see Section~\ref{S:ucfc} below) show that, for $F := \mathrm{H}^{n}(\mathbf{C}, \blank)$, we can set 
\[
D:=  \mathrm{H}_{n-1}(\mathbf{C}) \simeq w(\mathrm{H}^{n}(\mathbf{C}, \blank)).
\] 
While this formula may be viewed as being coordinate-free, it requires the use of a different functor, i.e.,  $\mathrm{H}_{n-1}(\mathbf{C}, \blank)$.

Recall that a direct summand of a covariant $\Ext$ need not be an $\Ext$, as was shown by Auslander in~\cite{A69}, refuting an earlier conjecture of his own. The conjecture is, however, true if the domain category of $F$ has denumerable sums~\cite{Fr66}, or if it is of finite global dimension~\cite[Theorem~4.8]{A66}, or if the fixed argument of $\Ext$ is artinian (or, more generally, is noetherian and has the descending chain condition for the images of nested endomorphisms)~\cite[Proposition~11]{M-2016}. In those cases, $F \simeq \Ext^{1}(D, \blank)$, and one may ask how to recover $D$ from the functor $\Ext^{1}(D, \blank)$.  As~$D$ is defined only up to projectively stable equivalence, one has to look for an answer in the stable category. In that category, the isomorphism type of $D$ is uniquely determined by the isomorphism type of the corresponding representable functor $(\underline{D, \blank})$, and this immediately leads to the desired solution:
\[
(\underline{D, \blank}) \simeq S_{1}\Ext^{1}(D, \blank) \simeq S_{1}F,\footnote{The reader who feels the given answer is not ``constructive'' enough for computational purposes may rethink what it means to compute and start thinking of computing with functors.}
\] 
where $S_{1}$ denotes the left satellite.
\end{remark}

\section{Universal coefficient theorems for cohomology}\label{S:ucfc}

As an illustration of~\eqref{D:7-term-cx-fp} and~\eqref{D:defect} we revisit the universal coefficient formula for cohomology (see also~\cite[Proposition 5.8 and Corollary 5.9]{A66}). Let $\Lambda$ be an arbitrary ring and $\mathbf{C}$ a complex of projectives 
\[
\ldots \lra \mathrm{C}_{n+1} \overset{d_{n+1}}\lra \mathrm{C}_{n} \overset{d_{n}}\lra \mathrm{C}_{n-1} \lra \ldots
\]
with projective boundaries. Thus the cycles are projective, too. We use the standard notation $\mathrm{B, C, Z}$ to denote, respectively, the boundaries, chains, and  cycles of~$\mathbf{C}$. Following~\cite{Bour}, we use the notation $\mathrm{Homgr}$ for the complex of homomorphisms between complexes. To wit, for complexes $(\mathbf{C}, d)$ and $(\mathbf{C}', d')$, the degree $t$ chains of $\mathrm{Homgr}(\mathbf{C}, \mathbf{C}')$ are the elements of $\prod_{i \in \Z}(C_{i}, C_{i+t})$. The differential is defined by $D(f) := d'f - (-1)^{t}fd$, where $f$ is a degree $t$ chain. Thus, $\mathrm{Homgr}$ is a bifunctor from complexes of $\Lambda$-modules to complexes of abelian groups. 

 Given an integer $n$ and viewing (say, left) modules as complexes concentrated in degree 0, we have an additive functor $F := \mathrm{H}^{n}(\mathrm{Homgr}(\mathbf{C}, \blank))$ on left modules with values in abelian groups. Since finitely presented functors are defined as cokernels of natural transformations between representables (see~\ref{SS:fpf}) and since the kernel and the cokernel of a natural transformation between finitely presented functors are also finitely presented~\cite[Proposition~3.1]{A82}, we have that  
 $\mathrm{H}^{n}(\mathrm{Homgr}\big(\mathbf{C}, \blank)\big)$ is always finitely presented. To avoid notational clutter, denote this functor by  $\mathrm{H}^{n}(\mathbf{C}, \blank)$. 

Applying the contravariant Yoneda embedding to the canonical exact sequence
\[
\begin{tikzcd}
 0 \ar[r]
 & \mathrm{H}_{n}(\mathbf{C}) \ar[r]
 & \mathrm{C}_{n}/\mathrm{B}_{n} \ar[r] \ar[dr, two heads]
 & \mathrm{Z}_{n-1} \ar[r]
 & \mathrm{H}_{n-1}(\mathbf{C}) \ar[r]
 & 0
\\
  &
  &
  & B_{n-1} \ar[u, tail]
\end{tikzcd}
\]
we have a left-exact four-term complex of additive functors in the middle row of the commutative diagram of natural transformations
\begin{equation}\label{D:ucfc}
\begin{gathered}
\begin{tikzcd}[cramped, sep=small]
 	&\Ext^{1}(\mathrm{H}_{n-1}(\mathbf{C}), \blank) \ar[dr, dashed]
\\
	& (\mathrm{B}_{n-1}, \blank) \ar[u, two heads, "\gamma"] \ar[dr, tail, "\delta"]
	& \Cok \delta\iota \ar[dr, dashed]
\\
	(\mathrm{H}_{n-1}(\mathbf{C}), \blank) \ar[r, tail]
	& (\mathrm{Z}_{n-1}, \blank) \ar[u, "\iota"] \ar[r, "\delta\iota"] 
	\ar[dr, "\mu\delta\iota"']
	& (\mathrm{C}_{n}/\mathrm{B}_{n}, \blank) \ar[r, two heads, "\varepsilon"] 
	\ar[d, tail, "\mu"]
	\ar[u, two heads]
	& (\mathrm{H}_{n}(\mathbf{C}), \blank)
\\
	&
	& (\mathrm{C}_{n}, \blank) \ar[d, "\nu"]
\\
	&
	& (\mathrm{B}_{n}, \blank)
\end{tikzcd}
\end{gathered}
\end{equation}
Notice that $(\delta, \varepsilon)$ is a split short exact sequence since 
$\mathrm{B}_{n-1}$ is projective by assumption. Likewise, $\gamma = \coker \iota$ since $\mathrm{Z}_{n-1}$ is projective. The dashed arrows arise from the 
circular sequence for~$\delta$ and~$\iota$ and yield a short exact sequence
\[
0 \lra \Ext^{1}(\mathrm{H}_{n-1}(\mathbf{C}), \blank) \lra \Cok \delta\iota \lra
 (\mathrm{H}_{n}(\mathbf{C}), \blank) \lra 0.
\]
From the circular sequence for $\mu$ and $\delta \iota$ we have a short exact sequence 
\[
0 \lra \Cok \delta \iota \lra \Cok \mu \delta  \iota \lra \Cok \mu \lra 0.
\]
By the left-exactness of the Hom functor (applied twice), we have 
\[
\Cok \mu \simeq  (\mathrm{C}_{n}, \blank)/\Imr \mu = 
(\mathrm{C}_{n}, \blank)/\Ker \nu = (\mathrm{C}_{n}, \blank)/\Ker (d_{n+1}, \blank).
\]
 On the other hand, since $B_{n-2}$ is projective, the induced natural transformation 
 $(\mathrm{C}_{n-1},\blank) \lra (\mathrm{Z}_{n-1}, \blank)$ is epic and therefore 
 \[
 \Cok \mu \delta \iota = \Cok (d_{n}, \blank) =  
 (\mathrm{C}_{n}, \blank)/\Imr (d_{n}, \blank).
 \]
 It now follows from the circular sequence for $\delta \iota$ and $\mu$ that $\Cok \delta \iota \simeq \mathrm{H}^{n}(\mathbf{C}, \blank)$, and we have a short exact sequence 
\begin{equation}\label{D:ucf-coh}
 0 \lra 
 \Ext^{1}(\mathrm{H}_{n-1}(\mathbf{C}), \blank) \lra 
 \mathrm{H}^{n}(\mathbf{C}, \blank) \lra
 (\mathrm{H}_{n}(\mathbf{C}), \blank) 
 \lra 0.
 \end{equation}
Because $\varepsilon$ is split, the same is true, by the commutativity of the rightmost triangle, for the last map in this sequence.\footnote{Since this splitting arises from the splittings inside the original complex $\mathbf{C}$, which could be chosen arbitrarily,  there is no reason for it be functorial in~$\mathbf{C}$. However, if the terms of the resulting formula are viewed as functors on the coefficients and not on complexes, then a functorial splitting appears because representable functors are projective objects in the functor category.} We also notice that this map is an isomorphism on injectives. Together with the fact that 
$(\mathrm{H}_{n}(\mathbf{C}), \blank)$ is left-exact, this yields  isomorphisms
\[
(\mathrm{H}_{n}(\mathbf{C}), \blank) \simeq R^{0}\mathrm{H}^{n}(\mathbf{C}, \blank)
\quad \textrm{and} \quad \overline{\mathrm{H}^{n}(\mathbf{C}, \blank)} \simeq
\Ext^{1}(\mathrm{H}_{n-1}(\mathbf{C}), \blank).
\]
Comparison of the above short exact sequence with the fundamental 
sequence~\eqref{D:7-term-cx} shows that $\beta_{F} : R^{0}F \lra \overline{F} \circ \Sigma$ is zero in this case. On the other hand, it is easy to construct complexes over rings of global dimension at least 2 such that 
$\overline{F} \circ \Sigma = \Ext^{2}(\mathrm{H}_{n-1}(\mathbf{C}), \blank)$ is not  zero.

Finally, we remark that, since representable functors are projective objects in the functor category, the sequence~\eqref{D:ucf-coh} splits. In other words, the classical universal coefficient formula splits functorially. 
\medskip

A different approach to universal coefficients for cohomology is based on the already mentioned fact that $\mathrm{H}^{n}(\mathbf{C}, \blank) \simeq \Cok \delta \iota$ is finitely presented. An explicit projective resolution of 
$\mathrm{H}^{n}(\mathbf{C}, \blank)$ is embedded in the diagram~\eqref{D:ucfc}:
 \[
 0 \lra (\mathrm{H}_{n-1}(\mathbf{C}), \blank) \lra (\mathrm{Z}_{n-1}, \blank) \lra  (\mathrm{C}_{n}/\mathrm{B}_{n}, \blank) \lra \mathrm{H}^{n}(\mathbf{C}, \blank) \lra 0.
 \]
It was produced under the assumption that $\mathbf{C}$ was a complex of projectives with projective boundaries. However, the cohomology functor is finitely presented without any assumptions on the complex. This fact allows, as we shall now see, to establish a universal coefficient theorem for cohomology in a completely  general context. Indeed, straight from the definition of cohomology, we have a projective resolution
\[
 0 \lra (\mathrm{C}_{n-1}/\mathrm{B}_{n-1}, \blank) \lra (\mathrm{C}_{n-1}, \blank) \lra  (\mathrm{C}_{n}/\mathrm{B}_{n}, \blank) \lra \mathrm{H}^{n}(\mathbf{C}, \blank) \lra 0
\]
of the $n$-th cohomology functor. According to~\eqref{D:defect}, this yields the defect of the cohomology functor and all of its right-derived functors.
\begin{lemma}
Let $\mathbf{C} : 
\ldots \lra \mathrm{C}_{n+1} \overset{d_{n+1}}\lra \mathrm{C}_{n} \overset{d_{n}}\lra \mathrm{C}_{n-1} \lra \ldots$ be an arbitrary complex. Then
 \[
w(\mathrm{H}^{n}(\mathbf{C}, \blank)) \simeq \Ker (\mathrm{C}_{n}/\mathrm{B}_{n} \to \mathrm{C}_{n-1}) \simeq \mathrm{H}_{n}(\mathbf{C}). 
\]
As a consequence, 
\[
R^{i}(\mathrm{H}^{n}(\mathbf{C}, \blank)) \simeq \Ext^{i}(\mathrm{H}_{n}(\mathbf{C}), \blank)
\]
for all integers $i$. \qed
\end{lemma}

By~\eqref{D:Fsub}, the Yoneda embedding of the short exact sequence 
\begin{equation}\label{D:F-inj-stab}
0 \lra \mathrm{B}_{n-1}  \lra \mathrm{C}_{n-1} \lra \mathrm{C}_{n-1}/\mathrm{B}_{n-1}  \lra 0
\end{equation}
yields an explicit projective resolution 
\begin{equation}\label{D:Fsub-res}
0 \lra (\mathrm{C}_{n-1}/\mathrm{B}_{n-1}, \blank) \lra (\mathrm{C}_{n-1}, \blank)
\lra (\mathrm{B}_{n-1}, \blank) \lra \overline{\mathrm{H}^{n}(\mathbf{C}, \blank}) \lra 0
\end{equation}
of the sub-stabilization of $\mathrm{H}^{n}(\mathbf{C}, \blank)$. Summarizing the above observations, we have
\begin{namedthm*}{Universal Coefficient Theorem for Cohomology of Arbitrary Complexes}
 Let $\mathbf{C}$ be an arbitrary complex. Then the right fundamental sequence for $\mathrm{H}^{n}(\mathbf{C}, \blank)$ is of the form
\begin{equation}\label{D:rfs-co-arb}
\begin{tikzcd}
 0 \ar[r] 
   & \overline{\mathrm{H}^{n}(\mathbf{C}, \blank}) \ar[r] 
      & \mathrm{H}^{n}(\mathbf{C}, \blank) \ar[r] 
                  \ar[d, phantom, ""{coordinate, name=Z}]
         & (\mathrm{H}_{n}(\mathbf{C}), \blank) \ar[dll, rounded corners,
                                   to path={ -- ([xshift=2ex]\tikztostart.east)
						|- (Z) [pos=0.3]\tikztonodes
						-| ([xshift=-2ex]\tikztotarget.west) 
						-- (\tikztotarget)}] 
\\
   & \overline{\mathrm{H}^{n}(\mathbf{C}, \blank}) \circ \Sigma \ar[r] 
      & S^{1} (\mathrm{H}^{n}(\mathbf{C}, \blank)) \ar[r]
      			  \ar[d, phantom, ""{coordinate, name=Y}]
         & \Ext^{1}(\mathrm{H}_{n}(\mathbf{C}), \blank) \ar[dll, rounded corners,
                                   to path={ -- ([xshift=2ex]\tikztostart.east)
						|- (Y) [near end]\tikztonodes
						-| ([xshift=-2ex]\tikztotarget.west) 
						-- (\tikztotarget)}] 
\\
   & \overline{\mathrm{H}^{n}(\mathbf{C}, \blank}) \circ \Sigma^{2} \ar[r] 
      & S^{2}(\mathrm{H}^{n}(\mathbf{C}, \blank))  \ar[r]
      			  \ar[d, phantom, ""{coordinate, name=X}]
         & \Ext^{2}(\mathrm{H}_{n}(\mathbf{C}), \blank) \ar[dll, rounded corners,
                                   to path={ -- ([xshift=2ex]\tikztostart.east)
						|- (X) [near end]\tikztonodes
						-| ([xshift=-2ex]\tikztotarget.west) 
						-- (\tikztotarget)}]
\\
   & \overline{\mathrm{H}^{n}(\mathbf{C}, \blank}) \circ \Sigma^{3} \ar[r] 
      & S^{3}(\mathrm{H}^{n}(\mathbf{C}, \blank)) \ar[r]
         & \ldots
\end{tikzcd}
\end{equation}
where $\overline{\mathrm{H}^{n}(\mathbf{C}, \blank})$ can be computed by~\eqref{D:Fsub-res}. This sequence is exact except, possibly, at the rightmost terms. If $\mathrm{H}^{n}(\mathbf{C}, \blank)$ is half-exact\footnote{This happens, for example, when $\mathbf{C}$ is a complex of projectives, see Lemma~\ref{L:stab=Ext} below.}, then this sequence is exact. \qed
\end{namedthm*}

Our next goal is to identify the satellites and the sub-stabilizations of the cohomology functor. For that, we will need an additional assumption, but first we make a general observation.

\begin{lemma}\label{L:S=stab}
Suppose $(F^{i})_{i \in \Z}$ is a cohomological functor, i.e., an exact 
$\delta$-functor \textup{(}\cite[2.1]{Gr57}\textup{)}.\footnote{In this lemma, our definition of exact $\delta$-functor is a little more general than the quoted one. To wit, the integer $n$ can vary over the entire $\Z$.} Then, for each integer $n$,  the natural transformation $\theta^{n} : S^{1}F^{n} \lra \overline{F^{n + 1}}$ induced by $\delta^{n}$ is an isomorphism.
\end{lemma}

\begin{proof}
 For a module $B$, choose an exact sequence $0 \to B \to I \to \Sigma B \to 0$ with an injective $I$, apply $F$, and pass to the corresponding long exact sequence. 
\end{proof}

\begin{lemma}\label{L:stab=Ext}
 Suppose $\mathbf{C}$ is a projective complex. Then $(\mathrm{H}^{i}(\mathbf{C}, \blank))_{i \in \Z}$ is a cohomological functor, and, for each integer $n$, the canonical natural transformation 
 \[
 \xi^{n-1} : \overline{\mathrm{H}^{n}(\mathbf{C}, \blank}) \lra \Ext^{1}(\mathrm{C}_{n-1}/\mathrm{B}_{n-1},\blank) 
 \]
 is an isomorphism.
\end{lemma}
 
\begin{proof}
 Since $\mathbf{C}$ is projective, applying the functor $(\mathbf{C}, \blank)$ to a short exact sequence of modules produces a short exact sequence of complexes. The first result now follows from the fact that cohomology is a cohomological functor on complexes.
 
 To prove the second assertion, apply the Yoneda embedding to the short exact sequence~\eqref{D:F-inj-stab}, pass to the cohomology long exact sequence, and compare the result with~\eqref{D:Fsub-res}.
\end{proof}

Now, we identify the satellites of $\mathrm{H}^{n}(\mathbf{C}, \blank)$. 

\begin{proposition}
 Suppose $\mathbf{C}$ is a projective complex. Then the natural transformation  
 $\theta^{n} : S^{1}(\mathrm{H}^{n}(\mathbf{C}, \blank)) \lra \overline{\mathrm{H}^{n+1}(\mathbf{C}, \blank})$ induced by $\delta^{n}$ is a functor isomorphism. The composition
 \[
 \xi^{n}\theta^{n} : S^{1}(\mathrm{H}^{n}(\mathbf{C}, \blank)) \lra 
 \Ext^{1}(\mathrm{C}_{n}/\mathrm{B}_{n},\blank)
 \] 
 is an isomorphism.
\end{proposition}

\begin{proof}
 By Lemma~\ref{L:stab=Ext}, $(\mathrm{H}^{i}(\mathbf{C}, \blank))_{i \in \Z}$ is a cohomological functor. The first assertion is then a consequence of Lemma~\ref{L:S=stab}. The second assertion follows from Lemma~\ref{L:stab=Ext}.
\end{proof}

Summarizing the foregoing discussion and utilizing dimension shift, we have 

\begin{namedthm*}{Universal Coefficient Theorem for Cohomology of Projective Complexes}
  Let $\mathbf{C}$ be a projective complex. Then the right fundamental sequence for $\mathrm{H}^{n}(\mathbf{C}, \blank)$ is exact and of the form
  \begin{equation}\label{D:rfs-co-proj}
\begin{tikzcd}[cramped, sep=small]
 0 \ar[r] 
   & \Ext^{1}(\mathrm{C}_{n-1}/\mathrm{B}_{n-1},\blank) \ar[r] 
      & \mathrm{H}^{n}(\mathbf{C}, \blank) \ar[r] 
                  \ar[d, phantom, ""{coordinate, name=Z}]
         & (\mathrm{H}_{n}(\mathbf{C}), \blank) \ar[dll, rounded corners,
                                   to path={ -- ([xshift=2ex]\tikztostart.east)
						|- (Z) [pos=0.3]\tikztonodes
						-| ([xshift=-2ex]\tikztotarget.west) 
						-- (\tikztotarget)}] 
\\
   & \Ext^{2}(\mathrm{C}_{n-1}/\mathrm{B}_{n-1},\blank) \ar[r] 
      & \Ext^{1}(\mathrm{C}_{n}/\mathrm{B}_{n},\blank) \ar[r]
      			  \ar[d, phantom, ""{coordinate, name=Y}]
         & \Ext^{1}(\mathrm{H}_{n}(\mathbf{C}), \blank) \ar[dll, rounded corners,
                                   to path={ -- ([xshift=2ex]\tikztostart.east)
						|- (Y) [near end]\tikztonodes
						-| ([xshift=-2ex]\tikztotarget.west) 
						-- (\tikztotarget)}] 
\\
   & \Ext^{3}(\mathrm{C}_{n-1}/\mathrm{B}_{n-1},\blank) \ar[r] 
      & \Ext^{2}(\mathrm{C}_{n}/\mathrm{B}_{n},\blank)  \ar[r]
      			  \ar[d, phantom, ""{coordinate, name=X}]
         & \Ext^{2}(\mathrm{H}_{n}(\mathbf{C}), \blank) \ar[dll, rounded corners,
                                   to path={ -- ([xshift=2ex]\tikztostart.east)
						|- (X) [near end]\tikztonodes
						-| ([xshift=-2ex]\tikztotarget.west) 
						-- (\tikztotarget)}]
\\
   & \Ext^{4}(\mathrm{C}_{n-1}/\mathrm{B}_{n-1},\blank) \ar[r] 
      & \Ext^{3}(\mathrm{C}_{n}/\mathrm{B}_{n},\blank)  \ar[r]
         & \ldots
\end{tikzcd}
\end{equation} \qed
\end{namedthm*}

\begin{remark}
 Yuya Otake informed me (October 30, 2024) he had also obtained the UCT for cohomology of projective complexes, see~\cite[Corollary 3.1]{O-25}. 
\end{remark}

\section{The left fundamental sequence of a covariant functor}\label{S:L0}

Suppose $\mathcal{A}$ has enough projectives. If $F : \mathcal{A} \lra \ab$ is an additive covariant functor, the \texttt{quot-stabilization} $\underline{F} : \mathcal{A} \lra \ab$ is defined as the cokernel of the canonical natural transformation 
$\lambda_{F} :  L_{0}F \lra F$, whose domain is the zeroth left-derived functor of $F$. Thus we have a defining exact sequence
\begin{equation}\label{E:first-Ldef-seq}
L_{0}F \overset{\lambda_{F}}\lra F \lra \underline{F} \lra 0.
\end{equation}

Since $\lambda_{F}$ is an isomorphism on projectives, $\underline{F}$ vanishes on projectives. This fact, in view of Schanuel's lemma, shows that $\underline{F} \circ \Omega$ is a well-defined functor, where~$\Omega$ denotes the operation of taking the kernel of the projection from a projective module onto a module. For brevity, we shall refer to that kernel as a syzygy module. 

Equivalently, given a $\Lambda$-module $B$, the value $\underline{F}(B)$ can be defined and computed by finding an epimorphism $P \overset{\pi}\lra B \lra 0$ with $P$ projective, 
applying $F$, and taking the cokernel. This yields a system of defining exact sequences
\begin{equation}\label{E:i-stab-via-ancestor}
 F(P) \overset{F(\pi)}\lra F(B) \overset{q}\lra \underline{F}(B) \lra 0.
\end{equation}

Similar to what we did in Section~\ref{S:left-fs-co}, we want to construct a natural transformation 
\[
\alpha_{F} : \underline{F} \circ \Omega \lra L_{0}F.
\]
Going back to the construction of the right fundamental sequence, we reverse the directions of the arrows, replace injectives with projectives, kernels with cokernels, right satellites with left satellites, etc., and obtain a ``left'' analog of the circular 
diagram~\eqref{D:circ-diag}

\begin{equation}\label{D:circ-diag-quot}
\begin{tikzcd}
 0 \ar[d]
 & 
 &
 & \underline{F}(B)
 & 0
\\
 K \ar[d, "a"'] \ar[dr, tail, "j"'] 
 & 
 &
 & F(B) \ar[r, dashleftarrow, "y"] \ar[u, twoheadrightarrow]
 & L \ar[u]
\\
 Z_{1} \ar[r, tail] \ar[ddr, "b"']
 & F(P_{1}) \ar[rd, "F(e)"'] \ar[rr, "F(d_{1})"]
 &
 & F(P_{0}) \ar[r, two heads] \ar[ru, "h"', two heads] \ar[u, "F(\pi)"]
 & L_{0}F(B) \ar[u, two heads, "l"']
\\
 &
 & F(\Omega B) \ar[ru, "F(m)"'] \ar[rd, two heads]
 \\
 & S_{1}F(B)  \ar[ru, tail] \ar[rr, "g"']
 &
 & \underline{F}(\Omega B) \ar[ruu, "\alpha_{F}(B)"']
\end{tikzcd}
\end{equation}
leading to an exact six-term circular sequence
\begin{equation}\label{D:6-term-es-left} 
0 \lra K \lra Z_{1} \lra S_{1}F(B) \lra \underline{F}(\Omega B) \lra L_{0}F(B) \lra L \lra 0,
\end{equation}
which in turn yields a seven-term sequence
\begin{equation}\label{D:7-term-cx-left}
\begin{tikzcd}[cramped, sep=small]
 0 \ar[r]
 & K \ar[r] 
 & Z_{1} \ar[r]  
 & S_{1}F(B)  \ar[r] 
 & \underline{F}(\Omega B) \ar[r]
 & L_{0}F(B) \ar[r, "yl"]
 & F(B) \ar[r, "q"]
 & \underline{F}(B) \ar[r] 
 & 0
 \end{tikzcd}
\end{equation}
and the desired natural transformation 
\[
\alpha_{F} : \underline{F} \circ \Omega \lra L_{0}F.
\]
Furthermore, arguments entirely analogous to those from Section~\ref{S:left-fs-co} lead to the following sequence of functors

\begin{equation}\label{D:lfs-co-fun}
\begin{tikzcd}
  \ldots \ar[r] 
      &  S_{3}F \ar[r]
      			  	\ar[d, phantom, ""{coordinate, name=Z}]
         & \underline{F}\circ \Omega^{3} \ar[dll, rounded corners,
                                   to path={ -- ([xshift=2ex]\tikztostart.east)
						|- (Z) [pos=0.3]\tikztonodes
						-| ([xshift=-2ex]\tikztotarget.west) 
						-- (\tikztotarget)}] 
\\
    L_{2}F \ar[r]   
      & S_{2}F \ar[r]
      			  	\ar[d, phantom, ""{coordinate, name=Y}]
         & \underline{F}\circ \Omega^{2} \ar[dll, rounded corners,
                                   to path={ -- ([xshift=2ex]\tikztostart.east)
						|- (Y) [near end]\tikztonodes
						-| ([xshift=-2ex]\tikztotarget.west) 
						-- (\tikztotarget)}] 
\\
   L_{1}F \ar[r]			
      &   S_{1}F	\ar[r]
      				\ar[d, phantom, ""{coordinate, name=X}]
         & \underline{F}\circ \Omega  \ar[dll, "\alpha_{F}", rounded corners,
                                   to path={ -- ([xshift=2ex]\tikztostart.east)
						|- (X) [pos=0.3]\tikztonodes
						-| ([xshift=-2ex]\tikztotarget.west) 
						-- (\tikztotarget)}]
\\
   L_{0}F \ar[r, "\lambda_{F}"]              	 
      & F	\ar[r]			
         &  \underline{F} \ar[r]
            & 0
\end{tikzcd}
\end{equation}


\begin{definition}
 We shall refer to~\eqref{D:lfs-co-fun} as the left fundamental sequence of $F$.
\end{definition}


We also have an analog of Theorem~\ref{T:co-fund-right}.

\begin{theorem}\hfill\label{T:co-fund-left}
\begin{enumerate}
 \item The left fundamental sequence~\eqref{D:lfs-co-fun} is a complex, whose homology is zero except possibly, at the terms $L_{i}F$.
  
 \item If $F$ is half-exact, then the sequence is exact. 
 
 \item If $F$ is mono-preserving, then $\alpha_{F}$ is monic.
 
 \item If $F$ is epi-preserving, then $\underline{F} = 0$ and $\lambda_{F}$ is epic.
 
 \item $\lambda_{F}$ is an isomorphism if and only if $F$ is right-exact.

\end{enumerate}
\end{theorem}

\begin{proof}
 Similar to the proof of Theorem~\ref{T:co-fund-right}.
\end{proof}

Similar to Corollary~\ref{C:connect-iso}, we have

\begin{corollary}
 Suppose $F$ is a left-exact functor. Then the canonical natural transformation  
 $\underline{F} \circ \Omega^{i+1} \lra  L_{i}F $ is an isomorphism for each $i \geq 1$.  \qed
\end{corollary}

\begin{remark}
 Assume $\mathscr{A}$ is the category of left $\Lambda$-modules over a ring 
 $\Lambda$. Recall that the canonical natural transformation $F(\Lambda) \otimes _{\blank} \lra F$ is an isomorphism on~$\Lambda$. If $F$ preserves direct sums, then this transformation is an isomorphism on all projectives, and therefore it is isomorphic to the zeroth left-derived transformation for $F$. In particular, $L_{i}F \simeq \Tor_{i}(F(\Lambda), \blank)$ for all $i$.
\end{remark}

\begin{example}
 Let $F := (A, \blank)$. It is easy to see that the quot-stabilization of $F$ is 
 $(A, \blank)$ modulo projectives, i.e., $\underline{F} \simeq (\underline{A, \blank})$ in the classical notation.\footnote{This is not a tautology but a fortuitous (if not unfortunate) coincidence; the underline notation for the Hom modulo projectives seems to predate the underline notation for the quot-stabilization. As we shall see later, the situation is dramatically different for contravariant functors.} The left fundamental sequence then becomes
\begin{equation}\label{E:co-4-term for hom}
\xymatrix
	{
	0 \ar[r]
	& (\underline{A, \blank}) \circ  \Omega \ar[r]^{\alpha_{F}}          
	& L_{0}(A, \blank) \ar[r]^{\lambda_{A}}
	& (A, \blank) \ar[r]
	& (\underline{A, \blank}) \ar[r]
	& 0.
	}
\end{equation}
Since $(A, \blank)$ is left-exact, this sequence is exact. 
\end{example}

\begin{corollary}
Suppose $\Lambda$ is left hereditary. Then $L_{0}(A, \blank) \simeq P(A, \blank)$, the functor of maps with domain $A$ factoring through projectives, and the sequence~\eqref{E:co-4-term for hom}  becomes
 \[
 \xymatrix
	{
	0 \ar[r]        
	& P(A, \blank) \ar[r]^{\lambda_{A}}
	& (A, \blank) \ar[r]
	& (\underline{A, \blank}) \ar[r]
	& 0.
	}
 \]
 In particular, $P(A, \blank)$ is an exact functor.
\end{corollary}
\begin{proof}
 Since the syzygy module of any module is projective, the first term of the fundamental sequence vanishes. Since the fundamental sequence is exact, we also
 have $L_{0}(A, \blank) \simeq P(A, \blank)$. To show that $P(A, \blank)$ is exact, we first notice that it preserves monomorphisms because it is a subfunctor of the left-exact functor $(A, \blank)$. Being a zeroth left-derived functor, it is also right-exact, and therefore exact.
\end{proof}

Returning to general rings, assume now that $A$ is finitely generated. The canonical natural transformation $A^{\ast} \otimes \blank \lra (A, \blank)$ is an isomorphism on $\Lambda$ and, since~$A$ is finitely generated, on any projective. Combined with the fact that $A^{\ast} \otimes \blank$ is right-exact, this shows that this transformation is isomorphic to $\lambda_{F}$ and the sequence turns into
\begin{equation}
\xymatrix
	{
	0 \ar[r]
	& (\underline{A, \blank}) \circ  \Omega \ar[r]^{\alpha_{F}}          
	& A^{\ast} \otimes \blank \ar[r]^{\lambda_{A}}
	& (A, \blank) \ar[r]
	& (\underline{A, \blank}) \ar[r]
	& 0.
	}
\end{equation}

Furthermore, if $A$ is finitely presented, then  
$(\underline{A, \blank}) \simeq \Tor_{1}(\Tr A, \blank)$\footnote{Here we sketch a quick proof. For a module $B$, choose a syzygy sequence $0 \to \Omega B \to Q \to B \to 0$ and tensor it with the finite presentation~\eqref{D:TrA} of 
$\Tr A$. Convert the $P_{i}^{\ast} \otimes \blank$ to $(P_{i}, \blank)$, and use the snake lemma.} and we recover Auslander's sequence~\eqref{D:tor}:
\begin{equation}
\xymatrix@R1pt@C15pt
	{
	0 \ar[r]
	& \Tor_{2}(\Tr A, \blank) \ar[r]
	& A^{\ast} \otimes \blank \ar[r]^{\lambda_{A}}
	& (A, \blank) \ar[r]
	& \Tor_{1}(\Tr A, \blank) \ar[r]
	& 0.
	}
\end{equation}

\section{Interlude: tensor-copresented functors}\label{S:tensor-cop}

The goal of this section is to set the stage for the universal coefficient theorem for homology of arbitrary complexes. To this end, we examine a class of functors which exhibit a behavior dual to that of finitely presented functors.

\begin{definition}
 An additive functor from (left) $\Lambda$-modules to abelian groups is said to be tensor-copresented (\textrm{TC}, for short) if it is isomorphic to the kernel of a natural transformation between univariate tensor product functors. In other words, $F$ is TC if it is defined by an exact sequence
 \[
 0 \lra F \lra A \otimes \blank \lra B \otimes \blank,
 \]
 where $A$ and $B$ are right $\Lambda$-modules.
\end{definition}

\begin{remark}
 The reader who is familiar with the Auslander-Gruson-Jensen (AGJ) duality, will immediately see that TC functors on left modules are precisely the AGJ transforms of finitely presented functors on right modules.\footnote{The functor $A \otimes \blank$ is not in general an injective object in a functor category. Indeed, suppose the ring $\Lambda$ is von Neumann regular but not semisimple. Then there is a nonsplit  exact sequence $0 \to A \to B \to C \to 0$, and it must be pure. Apply the tensor embedding to get a short exact sequence of functors $0 \to A \otimes \blank \to B \otimes \blank  \to C \otimes \blank  \to 0$ (because the original sequence is pure exact). If $A \otimes \blank$ were injective, then this sequence would split, which would imply that the original sequence is split, too, a contradiction.

Now we can also show that the AGJ transformation is not always exact on tensor-copresented functors. Start with the same sequence as before and apply the tensor embedding to get a short exact sequence in the functor category. Now assume the AGJ transformation is exact and apply it to the obtained sequence. Get a short exact sequence of representable functors. It must split because representables are projective. Now applying the AGJ transformation again, we have the split exact sequence of tensor product functors we started with. But then the original sequence must split, a contradiction.}
\end{remark}

\begin{lemma}
 Let $A$ be a right module and $G$ an additive functor from left modules to abelian groups. If $G$ preserves direct sums, then any natural transformation $\phi : A \otimes \blank \lra G$ is uniquely determined by its $\prescript{}{\Lambda}\Lambda$-component.
\end{lemma}
\begin{proof}
 It follows from the assumptions that $\phi$ is uniquely determined on free left modules. To see that the same is true for any left module $X$, take a free presentation of $X$, apply $\phi$, and use the naturality of $\phi$ together with the fact that the tensor product is right-exact.
\end{proof}

Given a homomorphism $f : A \to B$ of right modules, we have a natural transformation $f \otimes \blank : A \otimes \blank \lra B \otimes \blank$. This yields a map $\alpha : (A, B) \lra (A \otimes \blank, B \otimes \blank)$. On the other hand, restriction to the $\prescript{}{\Lambda}\Lambda$-components yields a map $\beta$ in the opposite direction. The following known result is now easily checked.

\begin{corollary}
 The maps $\alpha$ and $\beta$ are inverses of each other. \qed
\end{corollary}

As a consequence, the defining exact sequence for the TC functor $F$ can now be rewritten as
\[
0 \lra F \lra A \otimes \blank \overset{f \otimes \blankd} \lra B \otimes \blank,
\]
with $f : A \to B$ \textcolor{red}{uniquely determined by the copresentation}. 
Moreover, using the right-exactness of the tensor product, we have a tensor coresolution 
\[
0 \lra F \lra A \otimes \blank \overset{f \otimes \blankd}\lra B \otimes \blank \lra C \otimes \blank \lra 0
\]
of $F$, where $C := \Cok f$.

Suppose now that $F$ is given by the above data. Our goal is to produce a tensor copresentation of the quot-stabilization $\underline{F}$ of $F$. Let $K_{f} := \Ker (f : A \to B)$. The exact commutative diagram 
\[
\begin{tikzcd}
    0 \ar[r]
 & K_{f} \ar[r]
 & A \ar[rr, "f"] \ar[rd, twoheadrightarrow, "e"']
 & 
 & B \ar[r]
 & C \ar[r]
 & 0,
\\
&
&
&D \ar[ru, rightarrowtail, "m"']
\end{tikzcd}
\]
where the middle triangle is an epi-mono factorization of $f$, gives rise, after applying the tensor product functor, to another commutative diagram with exact rows and columns:
\begin{equation}\label{D:quotTC}
 \begin{tikzcd}[cramped, sep=small]
& K_{f} \otimes \blank \ar[d, "\lambda_{F}"'] \ar[r, equal]
& K_{f} \otimes \blank \ar[d]
\\
 0 \ar[r]
 & F \ar[r] \ar[dd, "p"']
 & A \otimes \blank \ar[rr, "f \otimes \blankd" near start] \ar[rd, twoheadrightarrow] \ar[dd, twoheadrightarrow, "e \otimes \blankd"']
 & 
 & B \otimes \blank \ar[r] \ar[dd, equal]
 & C \otimes \blank \ar[r] \ar[dd, equal]
 & 0
\\
&
&
& X \ar[ru, rightarrowtail] \ar[dd, equal]
\\
0 \ar[r]
 & K \ar[r]
 & D \otimes \blank \ar[rr,  "m \otimes \blankd" near start] \ar[rd, twoheadrightarrow]
 & 
 & B \otimes \blank \ar[r]
 & C \otimes \blank \ar[r]
 & 0
\\
&
&
& X \ar[ru, rightarrowtail]
\end{tikzcd}
\end{equation}
Here the two triangles are epi-mono factorizations, $K := \Ker (m \otimes \blank)$, the map $F \to K$ is induced by the middle square, and the map $\lambda_{F}$ is defined by the universal property of kernels. We are now ready to prove the main result of this section.
\begin{theorem}
 In the above diagram, $K \simeq \underline{F}$ and the bottom row is a tensor copresentation of $\underline{F}$.
\end{theorem}
\begin{proof}
 We only need to prove the first assertion. We view the first three columns and the maps between them as a short exact sequence of complexes. Since the third column is exact, the map between the first two is a quasi-isomorphism (i.e., induces an isomorphism on the homology). In particular, the first column is exact at $F$ and $p = \coker \lambda_{F}$. Furthermore, $\lambda_{F}$ is clearly an isomorphism on projectives. Together with the fact that  $K_{f} \otimes \blank$ is right-exact, it shows that $\lambda_{F}$ is the zeroth left-derived transformation for $F$, and therefore $K$ is the quot-stabilization of $F$.
\end{proof}

\section{Universal coefficient theorems for homology}\label{S:ucfh}
For another illustration of Theorem~\ref{T:co-fund-left} we look at the classical universal  coefficient formula for homology. Let $\Lambda$ be an arbitrary ring and $(\mathbf{C}, d)$ a complex of flat right $\Lambda$-modules with flat boundaries. We keep the notation from Section~\ref{S:ucfc}.  Given an integer $n$, we have an additive functor $F := \mathrm{H}_{n}(\mathbf{C} \otimes \blank)$ on left modules with values in abelian groups.
 
 Applying the tensor embedding to the canonical exact sequence
 \[
\begin{tikzcd}
 0 \ar[r]
 & \mathrm{H}_{n}(\mathbf{C}) \ar[r]
 & \mathrm{C}_{n}/\mathrm{B}_{n} \ar[r] \ar[dr, two heads]
 & \mathrm{Z}_{n-1} \ar[r]
 & \mathrm{H}_{n-1}(\mathbf{C}) \ar[r]
 & 0
\\
  &
  &
  & \mathrm{B}_{n-1} \ar[u, tail]
\end{tikzcd}
\]
we have a right-exact four-term complex of additive functors in the middle row of the commutative diagram of natural transformations 
\[
\begin{tikzcd}
& \mathrm{B}_{n} \otimes \blank \ar[d]
\\
& \mathrm{C}_{n} \otimes \blank \ar[d, two heads, "\pi"] \ar[rd, "\iota \delta \pi"]
\\
 \mathrm{H}_{n}(\mathbf{C}) \otimes \blank  \ar[r, tail, "\varepsilon"] \ar[dr, dashed]
 & \mathrm{C}_{n}/\mathrm{B}_{n} \otimes \blank \ar[r, "\iota \delta"] \ar[dr, two heads, "\delta"']
 & \mathrm{Z}_{n-1} \otimes \blank \ar[r, two heads]
 & \mathrm{H}_{n-1}(\mathbf{C}) \otimes \blank 
\\
 & \Ker \iota \delta \ar[u, tail] \ar[dr, dashed]
 & \mathrm{B}_{n-1} \otimes \blank \ar[u, "\iota"']
\\
 &
 & \Tor_{1}(\mathrm{H}_{n-1}(\mathbf{C}), \blank) \ar[u, tail, "\gamma"']
\end{tikzcd}
\]
Notice that $(\varepsilon, \delta)$ is a short exact sequence since $\mathrm{B}_{n-1}$ is flat by assumption. Likewise, $\gamma = \ker \iota$ by the flatness of $\mathrm{Z}_{n-1}$. The dashed arrows arise from the universal property of kernels. They are also part of the circular sequence for $\delta$ and~$\iota$, which yields a short exact sequence 
\[
0 \lra \mathrm{H}_{n}(\mathbf{C}) \otimes \blank \lra \Ker \iota \delta \lra 
\Tor_{1}(\mathrm{H}_{n-1}(\mathbf{C}), \blank) \lra 0.
\]
From the circular sequence for $\pi$ and $\iota \delta$ we have a short exact sequence 
\[
0 \lra \Ker \pi \lra \Ker \iota \delta \pi \lra \Ker  \iota \delta \lra 0.
\]
By the right-exactness of the tensor product (applied twice), we have 
$\Ker \pi = \Imr d_{n +1} \otimes \blank$. On the other hand, since $B_{n-2}$ is flat, the natural transformation $\mathrm{Z}_{n-1} \otimes \blank \lra
\mathrm{C}_{n-1} \otimes \blank$ is monic and therefore $\Ker \iota \delta \pi = 
\Ker d_{n} \otimes \blank$. It follows that 
$ \Ker \iota \delta = \mathrm{H}_{n}(\mathbf{C} \otimes \blank)$ and we have a short exact sequence of natural transformations 
\begin{equation}\label{D:ucf-homology}
0 \lra \mathrm{H}_{n}(\mathbf{C}) \otimes \blank \lra \mathrm{H}_{n}(\mathbf{C} \otimes \blank) \lra \Tor_{1}(\mathrm{H}_{n-1}(\mathbf{C}), \blank) \lra 0.
\end{equation}
The first map in this sequence is an isomorphism on $\Lambda$. Since 
$\mathrm{H}_{n}(\mathbf{C} \otimes \blank)$ preserves direct sums, it is also an isomorphism on projectives, which, together with the fact that 
$\mathrm{H}_{n}(\mathbf{C}) \otimes \blank$ is right-exact, shows that this map is the zeroth right-derived transformation for $\mathrm{H}_{n}(\mathbf{C} \otimes \blank)$. Therefore we have isomorphisms 
\[
\mathrm{H}_{n}(\mathbf{C}) \otimes \blank \simeq L_{0} \mathrm{H}_{n}(\mathbf{C} \otimes \blank) \quad \textrm{and} \quad \underline{\mathrm{H}_{n}(\mathbf{C} \otimes \blank)} \simeq \Tor_{1}(\mathrm{H}_{n-1}(\mathbf{C}), \blank).
\]
Comparison of this sequence with the fundamental 
sequence~\eqref{D:lfs-co-fun}, shows that $\alpha_{F} : \underline{F} \circ 
\Omega \lra L_{0}F$ is zero in this case. On the other hand, examples of complexes with zero differentials over rings of global dimension at least 2 show that 
$\underline{F} \circ \Omega = \Tor_{2}(\mathrm{H}_{n-1}(\mathbf{C}), \blank)$ need not be zero.

Assume now that $\Lambda$ is left hereditary and $\mathbf{C}$ is a complex of projectives. In this case, the cycles and the boundaries are projective and the previous arguments apply. Furthermore, the short exact sequence $(\varepsilon, \delta)$ becomes split and therefore~\eqref{D:ucf-homology} splits as well. 

Similar to the universal coefficient theorem for cohomology, there is an alternative approach to universal coefficients for homology.  While the former is based on the fact that cohomology is a finitely presented functor, which allows to immediately determine its right-derived functors, the latter is based on the facts that homology is tensor-copresented and the just mentioned isomorphism 
\[
\mathrm{H}_{n}(\mathbf{C}) \otimes \blank \simeq L_{0} \mathrm{H}_{n}(\mathbf{C} \otimes \blank)
\]
is valid without any assumptions on the complex $\mathbf{C}$. Hence, we have isomorphisms 
\[
L_{i} \mathrm{H}_{n}(\mathbf{C} \otimes \blank) \simeq \Tor_{i}\big(\mathrm{H}_{n}(\mathbf{C}), \blank\big)
\]
for all $i$. 

First, we determine a tensor copresentation for $\mathrm{H}_{n}(\mathbf{C} \otimes \blank)$. Applying the tensor embedding to the commutative diagram 
\[
\begin{tikzcd}
&
 & C_{n}/B_{n} \ar[dr]
  &
   & C_{n-1}/B_{n-1}
\\
 \ldots \ar[r]
 & C_{n} \ar[rr, "d_{n}"] \ar[rd, twoheadrightarrow] \ar[ur, twoheadrightarrow]
  &
   & C_{n-1} \ar[r] \ar[ur, twoheadrightarrow]
    & \ldots
\\
 B_{n} \ar[ur, rightarrowtail]
 &
  & B_{n-1} \ar[ur, rightarrowtail]
\end{tikzcd}
\]
we have a commutative diagram 
\[
\begin{tikzcd}[cramped, sep=small]
&
&
 & C_{n}/B_{n} \otimes \blank \ar[dr, "r_{n}"]
  &
   & C_{n-1}/B_{n-1} \otimes \blank
\\
 C_{n+1} \otimes \blank \ar[rr, "d_{n+1} \otimes \blankd"] \ar[rd, twoheadrightarrow]
 &
 & C_{n} \otimes \blank \ar[rr, "d_{n} \otimes \blankd"] \ar[rd, twoheadrightarrow] \ar[ur, "p_{n}", twoheadrightarrow]
  &
   & C_{n-1} \otimes \blank \ar[r] \ar[ur, "p_{n-1}", twoheadrightarrow]
    & \ldots
\\
 & B_{n} \otimes \blank \ar[ur, "j_{n}"']
 &
  & B_{n-1} \otimes \blank \ar[ur, "j_{n-1}"']
\end{tikzcd}
\]
with right-exact diagonals (by the right-exactness of the tensor product). Since $r_{n}$ 
and~$j_{n-1}$ have the same image, the sequence
\begin{equation}\label{D:head}
C_{n}/B_{n} \otimes \blank \overset{r_{n}}\lra C_{n-1} \otimes \blank \overset{p_{n-1}}\lra C_{n-1}/B_{n-1} \otimes \blank \lra 0
\end{equation}
is exact. We want to determine the kernel of $r_{n}$. The circular sequence of the pair $p_{n}, r_{n}$ yields a short exact sequence 
\[
0 \lra \Ker p_{n} \lra \Ker (d_{n} \otimes \blank) \lra \Ker r_{n} \lra 0
\]
Since the image of $j_{n}$ coincides with the image of $d_{n+1} \otimes \blank$, we have an isomorphism $\Ker r_{n} \simeq \mathrm{H}_{n}(\mathbf{C} \otimes \blank)$.
Taking into account~\eqref{D:head}, we have a tensor copresentation of 
$\mathrm{H}_{n}(\mathbf{C} \otimes \blank)$:
\begin{equation}\label{D:hom-copres}
   0 \lra \mathrm{H}_{n}(\mathbf{C} \otimes \blank) \lra C_{n}/B_{n} \otimes \blank \overset{r_{n}}\lra C_{n-1} \otimes \blank \overset{p_{n-1}}\lra C_{n-1}/B_{n-1} \otimes \blank \lra 0.
\end{equation}
Since the image of the map $C_{n}/B_{n} \to C_{n-1}$ induced by the differential is $B_{n-1}$, we have, by~\eqref{D:quotTC}, a tensor copresentation of $\underline{\mathrm{H}_{n}(\mathbf{C} \otimes \blank)}$
\begin{equation}\label{D:q-stab-hom}
0 \lra \underline{\mathrm{H}_{n}(\mathbf{C} \otimes \blank)} \lra B_{n-1} \otimes \blank \overset{j_{n-1}}\lra C_{n-1} \otimes \blank \overset{p_{n-1}}\lra C_{n-1}/B_{n-1} \otimes \blank \lra 0.
\end{equation}
In other words, the tensor copresentation of $\underline{\mathrm{H}_{n}(\mathbf{C} \otimes \blank)}$ is given by the degree $n-1$ diagonal in the foregoing commutative diagram. 

In summary, we arrive at

\begin{namedthm*}{Universal Coefficient Theorem for Homology of Arbitrary Complexes}
Let $\mathbf{C}$ be an arbitrary complex. Then the left fundamental sequence for 
$\mathrm{H}_{n}(\mathbf{C} \otimes \blank)$ is of the form

\begin{equation}\label{D:lfs-ho-fun}
\begin{tikzcd}[cramped, sep=small]
  \ldots \ar[r] 
      &  S_{3}\mathrm{H}_{n}(\mathbf{C} \otimes \blank) \ar[r]
      			  	\ar[d, phantom, ""{coordinate, name=Z}]
         & \underline{\mathrm{H}_{n}(\mathbf{C} \otimes \blank)}\circ \Omega^{3} \ar[dll, rounded corners,
                                   to path={ -- ([xshift=2ex]\tikztostart.east)
						|- (Z) [pos=0.3]\tikztonodes
						-| ([xshift=-2ex]\tikztotarget.west) 
						-- (\tikztotarget)}] 
\\
   \Tor_{2}\big(\mathrm{H}_{n}(\mathbf{C}), \blank\big) \ar[r]   
      & S_{2}\mathrm{H}_{n}(\mathbf{C} \otimes \blank) \ar[r]
      			  	\ar[d, phantom, ""{coordinate, name=Y}]
         & \underline{\mathrm{H}_{n}(\mathbf{C} \otimes \blank)}\circ \Omega^{2} \ar[dll, rounded corners,
                                   to path={ -- ([xshift=2ex]\tikztostart.east)
						|- (Y) [near end]\tikztonodes
						-| ([xshift=-2ex]\tikztotarget.west) 
						-- (\tikztotarget)}] 
\\
  \Tor_{1}\big(\mathrm{H}_{n}(\mathbf{C}), \blank\big) \ar[r]			
      &   S_{1}\mathrm{H}_{n}(\mathbf{C} \otimes \blank)	\ar[r]
      				\ar[d, phantom, ""{coordinate, name=X}]
         & \underline{\mathrm{H}_{n}(\mathbf{C} \otimes \blank)}\circ \Omega  \ar[dll, 	
         rounded corners,
                                   to path={ -- ([xshift=2ex]\tikztostart.east)
						|- (X) [pos=0.3]\tikztonodes
						-| ([xshift=-2ex]\tikztotarget.west) 
						-- (\tikztotarget)}]
\\
   \mathrm{H}_{n}(\mathbf{C}) \otimes \blank \ar[r]              	 
      & \mathrm{H}_{n}(\mathbf{C} \otimes \blank) \ar[r]			
         &  \underline{\mathrm{H}_{n}(\mathbf{C} \otimes \blank)} \ar[r]
            & 0,
\end{tikzcd}
\end{equation}
\noindent where $ \underline{\mathrm{H}_{n}(\mathbf{C} \otimes \blank)}$ can be computed by~\eqref{D:hom-copres}. This sequence is exact except, possibly, at the leftmost terms. If $\mathrm{H}_{n}(\mathbf{C} \otimes \blank)$ is half-exact\footnote{This happens, for example, when $\mathbf{C}$ is a complex of flats, see Lemma~\ref{L:flat} below.}, then this sequence is exact.  \qed
\end{namedthm*}

Our next goal is to provide more detailed information about the right satellites and the quot-stabilizations of the homology functor under the assumption that $\mathbf{C}$ is a flat complex. The next three results are entirely analogous to their cohomological counterparts from Section~\ref{S:ucfc}, and we omit their proofs.

We begin with a general observation.
\begin{lemma}
Suppose $(F_{i})_{i \in \Z}$ is a (co)homological functor, i.e., an exact $\delta$-functor. Then, for each integer $n$,  the natural transformation $\eta_{n} : \underline{F_{n+1}} \lra S_{1}F_{n}$ induced by~$\delta_{n+1}$ is an isomorphism. \qed
\end{lemma}

\begin{lemma}\label{L:flat}
 Suppose $\mathbf{C}$ is a flat complex. Then $(\mathrm{H}_{i}(\mathbf{C} \otimes \blank))_{i \in \Z}$ is cohomological functor, and, for each integer $n$, the canonical natural transformation 
 \[
 \tau_{n-1} :  \Tor_{1}(\mathrm{C}_{n-1}/\mathrm{B}_{n-1},\blank) \lra 
 \underline{\mathrm{H}_{n}(\mathbf{C} \otimes \blank})
 \]
 is an isomorphism. \qed
\end{lemma}

\begin{proposition}
 Suppose $\mathbf{C}$ is a flat complex. Then the natural transformation  
 $ \eta_{n} : \underline{\mathrm{H}_{n+1}(\mathbf{C} \otimes \blank}) \lra S_{1}(\mathrm{H}_{n}(\mathbf{C} \otimes \blank))$ induced by $\delta_{n+1}$   is a functor isomorphism.  The composition
 \[
 \eta_{n} \tau_{n} :  \Tor_{1}(\mathrm{C}_{n}/\mathrm{B}_{n},\blank) \lra
 S_{1}(\mathrm{H}_{n}(\mathbf{C} \otimes \blank))
 \]
 is an isomorphism. \qed
 \end{proposition}
 
 Summarizing the foregoing discussion and utilizing dimension shift, we have
 
\begin{namedthm*}{Universal Coefficient Theorem for Homology of Flat Complexes}
Let $\mathbf{C}$ be a flat complex. Then the left fundamental sequence for 
$\mathrm{H}_{n}(\mathbf{C} \otimes \blank)$ is exact and of the form
\begin{equation}
\begin{tikzcd}[cramped, sep=small]
  \ldots \ar[r] 
      &  \Tor_{3}(\mathrm{C}_{n}/\mathrm{B}_{n},\blank) \ar[r]
      			  	\ar[d, phantom, ""{coordinate, name=Z}]
         & \Tor_{4}(\mathrm{C}_{n-1}/\mathrm{B}_{n-1},\blank) \ar[dll, rounded corners,
                                   to path={ -- ([xshift=2ex]\tikztostart.east)
						|- (Z) [pos=0.3]\tikztonodes
						-| ([xshift=-2ex]\tikztotarget.west) 
						-- (\tikztotarget)}] 
\\
   \Tor_{2}\big(\mathrm{H}_{n}(\mathbf{C}), \blank\big) \ar[r]   
      & \Tor_{2}(\mathrm{C}_{n}/\mathrm{B}_{n},\blank) \ar[r]
      			  	\ar[d, phantom, ""{coordinate, name=Y}]
         & \Tor_{3}(\mathrm{C}_{n-1}/\mathrm{B}_{n-1},\blank) \ar[dll, rounded corners,
                                   to path={ -- ([xshift=2ex]\tikztostart.east)
						|- (Y) [near end]\tikztonodes
						-| ([xshift=-2ex]\tikztotarget.west) 
						-- (\tikztotarget)}] 
\\
  \Tor_{1}\big(\mathrm{H}_{n}(\mathbf{C}), \blank\big) \ar[r]			
      &   \Tor_{1}(\mathrm{C}_{n}/\mathrm{B}_{n},\blank)	\ar[r]
      				\ar[d, phantom, ""{coordinate, name=X}]
         & \Tor_{2}(\mathrm{C}_{n-1}/\mathrm{B}_{n-1},\blank)  \ar[dll, 	
         rounded corners,
                                   to path={ -- ([xshift=2ex]\tikztostart.east)
						|- (X) [pos=0.3]\tikztonodes
						-| ([xshift=-2ex]\tikztotarget.west) 
						-- (\tikztotarget)}]
\\
   \mathrm{H}_{n}(\mathbf{C}) \otimes \blank \ar[r]              	 
      & \mathrm{H}_{n}(\mathbf{C} \otimes \blank) \ar[r]			
         &  \Tor_{1}(\mathrm{C}_{n-1}/\mathrm{B}_{n-1},\blank) \ar[r]
            & 0
\end{tikzcd}
\end{equation}
\qed
\end{namedthm*} 

\begin{remark}
 Yuya Otake informed me (October 30, 2024) he had also obtained the UCT for homology of flat complexes, see~\cite[Corollary 3.2]{O-25}. 
\end{remark}

\section{The right fundamental sequence of a contravariant functor}\label{S:right-fs-contra}

Suppose $\mathcal{A}$ has enough projectives. For a contravariant functor $F : \mathcal{A} \lra \ab$, its \texttt{sub-stabilization} 
$\overline{F} : \mathcal{A} \lra \ab$ is defined as the kernel of the canonical natural transformation $\rho_{F} : F \lra R_{0}F$, whose codomain is the zeroth right-derived functor of $F$. Thus we have a defining exact sequence
\begin{equation}
0 \lra \overline{F} \lra F \overset{\rho_{F}}\lra R_{0}F,
\end{equation}
where the use of a subscript (as opposed to a superscript) simply indicates 
that~$F$ is contravariant. Since $\rho_{F}$ is an isomorphism on projectives, $\overline{F}$ vanishes on projectives. In fact, $\overline{F}$ is the largest subfunctor of $F$ with this property. It follows from Schanuel's lemma, that 
$\overline{F} \circ \Omega$, where~$\Omega$ denotes the syzygy operation, is a well-defined functor. 

Equivalently, given a $\Lambda$-module $B$, $\overline{F}(B)$ can be computed (at least, theoretically) by the following procedure. Cover $B$ by a projective ancestor $P \overset{\pi} \lra B \lra 0$, apply~$F$, and take the kernel. Thus we have defining exact sequences
\begin{equation}
 0 \lra \overline{F}(B) \lra {F}(B)  \overset{F(\pi)}\lra F(P).
\end{equation}

Similar to Section~\ref{S:left-fs-co}, we construct a natural transformation 
$\beta_{F} : R_{0}F \lra \overline{F} \circ \Omega$. To define its value on $B$, start with a projective resolution $\ldots \to P_{1} \overset{d_{1}} \to P_{0} \overset{\pi} \to B \to 0$, and let 
\[
\begin{tikzcd}
	P_{1}  \ar[rd, "e"'] \ar[rr, "d_{1}"]
	&
	& P_{0} 
\\
	& \Omega B \ar[ru, "m"']
\end{tikzcd}
\]
be an epi-mono factorization of $d_{1}$. Applying $F$, we have a commutative diagram 
\[
\begin{tikzcd}
 	R_{0}F(B) \ar[r, tail] \ar[d, dashrightarrow, "\beta_{F}(B)"']
	& F(P_{0}) \ar[r, "F(d_{1})"] \ar[d, "F(m)"]
	& F(P_{1}) \ar[d, equal]
\\
	\overline{F}(\Omega B) \ar[r, tail] \ar[dr]
	& F(\Omega B) \ar[r, "F(e)"'] \ar[d, two heads]
	& F(P_{1})
\\
	& S^{1}F(B)
\end{tikzcd}
\]
with exact rows, which gives rise to the desired natural transformation 
$\beta_{F}$. Similar to what we did in Section~\ref{S:left-fs-co}, we combine it with the defining sequence for $\overline{F}$ and extend the result to the right, 
thus obtaining a sequence
\begin{equation}\label{D:rfs-contra-fun}
\begin{tikzcd}
 0 \ar[r] 
   & \overline{F} \ar[r] 
      & F \ar[r, "\rho_{F}"] 
                  \ar[d, phantom, ""{coordinate, name=Z}]
         & R_{0}F \ar[dll, "\beta_{F}", rounded corners,
                                   to path={ -- ([xshift=2ex]\tikztostart.east)
						|- (Z) [pos=0.3]\tikztonodes
						-| ([xshift=-2ex]\tikztotarget.west) 
						-- (\tikztotarget)}] 
\\
   & \overline{F} \circ \Omega \ar[r] 
      & S^{1}F \ar[r]
      			  \ar[d, phantom, ""{coordinate, name=Y}]
         & R_{1}F \ar[dll, rounded corners,
                                   to path={ -- ([xshift=2ex]\tikztostart.east)
						|- (Y) [near end]\tikztonodes
						-| ([xshift=-2ex]\tikztotarget.west) 
						-- (\tikztotarget)}] 
\\
   & \overline{F} \circ \Omega^{2} \ar[r] 
      & S^{2}F  \ar[r]
      			  \ar[d, phantom, ""{coordinate, name=X}]
         & R_{2}F \ar[dll, rounded corners,
                                   to path={ -- ([xshift=2ex]\tikztostart.east)
						|- (X) [near end]\tikztonodes
						-| ([xshift=-2ex]\tikztotarget.west) 
						-- (\tikztotarget)}]
\\
   & \overline{F} \circ \Omega^{3} \ar[r] 
      & S^{3}F   \ar[r]
         & \ldots
\end{tikzcd}
\end{equation}

\begin{definition}
 We shall refer to~\eqref{D:rfs-contra-fun} as the right fundamental sequence of $F$.
\end{definition}

Accordingly, we have an analog of Theorem~\ref{T:co-fund-right}.

\begin{theorem}\hfill\label{T:contra-4-term-rho}
\begin{enumerate}
 \item\label{I:ends} The right fundamental sequence~\eqref{D:rfs-contra-fun} is a complex, whose homology is zero, except, possibly, at the terms   $R_{i}F$.
 
 \item If $F$ is half-exact, then the sequence is exact.
 
 \item If $F$ carries monomorphisms to epimorphisms, then $\beta_{F}$ is epic.
  
 \item If $F$ carries epimorphisms to monomorphisms, then $\overline{F} = 0$ and $\rho_{F}$ is monic.
 
 \item\label{I:lex} $\rho_{F}$ is an isomorphism if and only if $F$ is left-exact.

\end{enumerate}
\end{theorem}

\begin{proof}
 Entirely analogous to that of Theorem~\ref{T:co-fund-right}.
\end{proof}

As a trivial example, we may take $F := (\blank, A)$. Then $\overline{F} = 0$, and $\rho_{F}$ and all the remaining comparison maps from the satellites to the derived functors are isomorphisms.

Similar to Corollary~\ref{C:connect-iso}, we have

\begin{corollary}
 If $F$ is a right-exact functor, then the canonical natural transformation 
 $R_{i}F \lra \overline{F} \circ \Omega^{i+1}$ is an isomorphism for each $i \geq 1$. \qed
\end{corollary}

\begin{remark}
Assume $\mathscr{A}$ is the category of left $\Lambda$-modules over a ring 
 $\Lambda$. Recall that for any contravariant functor $F$, the canonical natural transformation $ F \to \big(\blank, F(\Lambda)\big)$ is an isomorphism on $\Lambda$. If $F$ carries coproducts to products, then this natural transformation is an isomorphism on all projectives, and therefore it is isomorphic to~$\rho_{F}$. Hence, the canonical natural transformation 
\[
R_{i}\rho_{F} : R_{i}F \lra \Ext^{i}(\blank, F(\Lambda))
\] 
is an isomorphism for all $i$, and each row of~\eqref{D:rfs-contra-fun} rewrites as 
\[
\overline{F} \circ \Omega^{i} \lra S^{i}F \lra \Ext^{i}(\blank, F(\Lambda)).
\]
Recall also that if $F$ is finitely presented, then the covariant defect $v(F)$ 
of~$F$ is defined as $F(\Lambda)$. If a finite presentation of $F$ arises as the Yoneda embedding of $f : B \lra A$, then $v(F) \simeq \Cok f \simeq F(\Lambda)$, and each row of the left fundamental sequence is of the form
\[
\overline{F} \circ \Omega^{i} \lra S^{i}F \lra \Ext^{i}(\blank, v(F)).
\]
\end{remark}

\begin{corollary}
Suppose $F$ is finitely presented and that $\Lambda$ is left hereditary. Then the first row of the right fundamental sequence becomes
\begin{equation}
\begin{tikzcd}
	0 \ar[r]
	& \overline{F} \ar[r]
	& F \ar[r, "\rho_{F}"]
	& \big(\blank, v(F)\big) \ar[r]
	& 0.
\end{tikzcd}
\end{equation}
If $F$ is half-exact, this sequence is exact, in fact, split exact.
\end{corollary}
\begin{proof}
Since finitely presented contravariant functors carry coproducts to products,  only the last claim needs a proof, but this is immediate since representable functors are projective objects.
\end{proof}

\section{The left fundamental sequence of a contravariant functor}\label{S:left-fs-contra}

Suppose $\mathcal{A}$ has enough injectives. For a contravariant functor $F : \mathcal{A} \lra \ab$, its \texttt{quot-stabilization} 
$\underline{F} : \mathcal{A} \lra \ab$ is defined as the cokernel of the canonical natural transformation $\lambda_{F} : L^{0}F \lra F$, whose domain is the zeroth left-derived functor of~$F$. Thus we have a defining exact sequence
\begin{equation}
L^{0}F  \overset{\lambda_{F}} \lra F \lra \underline{F} \lra 0
\end{equation}
where the use of a superscript (as opposed to a subscript) simply indicates 
that~$F$ is contravariant. Since $\lambda_{F}$ is an isomorphism on injectives, 
$\underline{F}$ vanishes on injectives. In fact, $\overline{F}$ is the largest quotient functor of $F$ with this property. It follows from Schanuel's lemma, that 
$\underline{F} \circ \Sigma$, where~$\Sigma$ denotes the cosyzygy operation in an injective resolution, is a well-defined functor. 

Equivalently, given a $\Lambda$-module $B$, $\underline{F}(B)$ can be computed (at least, theoretically) by the following procedure. Embed $B$ in an injective container $0 \lra B  \overset{\iota} \lra I$, apply~$F$, and take the cokernel of $F(\iota)$. Thus we have defining exact sequences
\begin{equation}
 F(I) \overset{F(\iota)} \lra F(B)\lra \underline{F}(B) \lra 0.
\end{equation}

Similar to Section~\ref{S:L0}, we construct a natural transformation 
\[
\alpha_{F} : \underline{F} \circ \Sigma \lra L^{0}F,
\]
and, using circular sequences, eventually obtain a sequence of functors

\begin{equation}\label{D:lfs-contra-fun}
\begin{tikzcd}
  \ldots \ar[r] 
      &  S_{3}F \ar[r]
      			  	\ar[d, phantom, ""{coordinate, name=Z}]
         & \underline{F}\circ \Sigma^{3} \ar[dll, rounded corners,
                                   to path={ -- ([xshift=2ex]\tikztostart.east)
						|- (Z) [pos=0.3]\tikztonodes
						-| ([xshift=-2ex]\tikztotarget.west) 
						-- (\tikztotarget)}] 
\\
    L^{2}F \ar[r]   
      & S_{2}F \ar[r]
      			  	\ar[d, phantom, ""{coordinate, name=Y}]
         & \underline{F}\circ \Sigma^{2} \ar[dll, rounded corners,
                                   to path={ -- ([xshift=2ex]\tikztostart.east)
						|- (Y) [near end]\tikztonodes
						-| ([xshift=-2ex]\tikztotarget.west) 
						-- (\tikztotarget)}] 
\\
   L^{1}F \ar[r]			
      &   S_{1}F	\ar[r]
      				\ar[d, phantom, ""{coordinate, name=X}]
         & \underline{F}\circ \Sigma  \ar[dll, "\alpha_{F}", rounded corners,
                                   to path={ -- ([xshift=2ex]\tikztostart.east)
						|- (X) [pos=0.3]\tikztonodes
						-| ([xshift=-2ex]\tikztotarget.west) 
						-- (\tikztotarget)}]
\\
   L^{0}F \ar[r, "\lambda_{F}"]              	 
      & F	\ar[r]			
         &  \underline{F} \ar[r]
            & 0
\end{tikzcd}
\end{equation}

\begin{definition}
 We shall refer to~\eqref{D:lfs-contra-fun} as the left fundamental sequence of $F$.
\end{definition}

\begin{theorem}\hfill \label{T:contra-fund-left}
\begin{enumerate}
 \item The left fundamental sequence~\eqref{D:lfs-contra-fun} is a complex, whose homology is zero except possibly, at the terms $L^{i}F$.
  
 \item If $F$ is half-exact, then the sequence is exact. 
 
 \item If $F$ carries epimorphisms to monomorphisms, then $\alpha_{F}$ is monic.
 
 \item If $F$ carries monomorphisms to epimorphisms, then $\underline{F} = 0$ and $\lambda_{F}$ is epic.
 
 \item $\lambda_{F}$ is an isomorphism if and only if $F$ is right-exact.

\end{enumerate}
\end{theorem}

\begin{proof}
 Similar to the proof of Theorem~\ref{T:co-fund-left}.
\end{proof}

Similar to Corollary~\ref{C:connect-iso}, we have

\begin{corollary}
 Suppose $F$ is a left-exact functor. Then the canonical natural transformation  
 $\underline{F} \circ \Sigma^{i+1} \lra  L^{i}F $ is an isomorphism for each $i \geq 1$.  \qed
\end{corollary}

\section{Stabilization and the exact Hom functors}\label{S:AR-form}

Before discussing how stabilization interacts with exact $\Hom$ functors, it is convenient to introduce the following notation. In any category, the covariant $\Hom$ functor $(A, \blank)$ will be denoted 
by $h^{A}$. Similarly, the contravariant $\Hom$ functor $(\blank, B)$ will be denoted by $h_{B}$. In the case of a category modulo projectives (respectively, injectives) we shall use the symbol 
$\underline{h}$ (respectively,~$\overline{h}$).

Suppose $\mathcal{B}$ and $\mathcal{C}$ are abelian categories with $\mathcal{B}$ having enough injectives. Also assume that 
$\mathcal{C}$ has an injective object $J$ and let $h_{J} : = (\blank, J)$. Let $F: \mathcal{B} \to \mathcal{C}$ be a functor admitting a right adjoint $G: \mathcal{C} \to \mathcal{B}$. 

Given $B \in  \mathcal{B}$ we choose an injective container
$0 \to B \to I$ and compute the injective stabilization of $F$ at $B$ by the exact sequence
\[
0 \lra \overline{F}(B) \lra F(B) \lra F(I).
\] 
Applying the exact functor $h_{J}$ we have the exact sequence
\[
(F(I), J) \lra (F(B), J) \lra (\overline{F}(B), J) \lra 0.
\]
Rewriting the first two terms by the adjoint property, we have the exact sequence
\[
(I, G(J)) \lra (B, G(J)) \lra (\overline{B, G(J)}) \lra 0,
\]
where overline denotes Hom modulo injectives. Comparing the last terms, we have
\begin{proposition}
There is a functor isomorphism $h_{J} \circ \overline{F} \simeq \overline{h}_{G(J)}$.
\qed
\end{proposition}

Now we want to show that the original Auslander-Reiten formula is a particular case of the just proved result. To this end, assume that $\Lambda$ contains a commutative subring $R$. We don't impose any restrictions on $\Lambda$ or $R$. Let $A$ be a right $\Lambda$-module and $B$ a left $\Lambda$-module. Let $J$ be an injective 
$R$-module and  $D_{J} : =  \prescript{}{R}(\blank, J)$. Setting $F := A \otimes _{\Lambda} \blank$, we have a functor from left $\Lambda$ modules to $R$-modules. Setting $G : = \prescript{}{R}(A, \blank)$, we have a functor from $R$-modules to left $\Lambda$-modules. When $A$ is finitely presented, the sub-stabilization of $F$ is given by $\Ext^{1}(\Tr A, \blank)$. Finally, replacing $A$ by $\Tr A$, we write the  isomorphism from the proposition in the form
\[
D_{\mathbf{J}}\Ext^{1}_{\Lambda}(A,B) \simeq \prescript{}{\Lambda} (\overline{B, D_{\mathbf{J}}\Tr A}),
\]
thus recovering the original Auslander-Reiten formula. In contrast with that formula, the isomorphism in the proposition  can be evaluated on arbitrary modules, finite or infinite.

Next we shall establish a result formally dual to the one just proved. Suppose $\mathcal{C}$ and $\mathcal{D}$ are abelian categories with $\mathcal{D}$ having enough projectives. Also assume that 
$\mathcal{C}$ has a projective object $Q$ and let 
$h^{Q} : = (Q, \blank)$. Let $F: \mathcal{C} \to \mathcal{D}$ be a functor admitting a right adjoint $G: \mathcal{D} \to \mathcal{C}$.

Given $A \in  \mathcal{D}$ we choose a projective ancestor
$P \to A \to 0$ and compute the projective stabilization of $G$ at $A$ by the exact sequence
\[
G(P) \lra G(A) \lra \underline{G}(A) \lra 0.
\] 
Applying the exact functor $h^{Q}$ we have the exact sequence
\[
(Q, G(P)) \lra (Q, G(A)) \lra (Q, \underline{G}(A)) \lra 0.
\] 
Rewriting the first two terms by the adjoint property, we have the exact sequence
\[
(F(Q), P) \lra (F(Q), A) \lra (\underline{F(Q), A}) \lra 0,
\]
where underline denotes Hom modulo projectives. Comparing the last terms, we have
\begin{proposition}\label{P:dual}
There is a functor isomorphism $h^{Q} \circ \underline{G} \simeq \underline{h}^{F(Q)}$. \qed
\end{proposition}

\end{document}